\documentclass[12pt]{amsart}

\usepackage[text={14.06cm,22.84cm},centering]{geometry}

\usepackage{color}
\usepackage{esint,amssymb}
\usepackage{graphicx}
\usepackage{MnSymbol}
\usepackage{mathtools}
\usepackage[colorlinks=true, pdfstartview=FitV, linkcolor=blue, citecolor=blue, urlcolor=blue,pagebackref=false]{hyperref}
\usepackage{microtype}

\usepackage{bm}
\usepackage{scalerel} 
\usepackage{dsfont}
\usepackage{mathrsfs}
\usepackage[font={footnotesize}]{caption}

\definecolor{darkgreen}{rgb}{0,0.5,0}
\definecolor{darkblue}{rgb}{0,0,0.7}
\definecolor{darkred}{rgb}{0.9,0.1,0.1}
\definecolor{lightblue}{rgb}{0,0.51,1}

\newtheorem{theorem}{Theorem}
\newtheorem{proposition}{Proposition}
\newtheorem{lemma}[proposition]{Lemma}
\newtheorem{corollary}[proposition]{Corollary}

\theoremstyle{remark}
\newtheorem{remark}[proposition]{Remark}

\theoremstyle{definition}
\newtheorem{definition}[proposition]{Definition}

\numberwithin{equation}{section}
\numberwithin{proposition}{section}

\newcommand{\Z}{\mathbb{Z}}
\newcommand{\N}{\mathbb{N}}
\newcommand{\R}{\mathbb{R}}

\renewcommand{\P}{\mathcal{P}}
\newcommand{\T}{\mathcal{T}}
\renewcommand{\H}{\mathcal{H}}

\newcommand{\Rd}{{\mathbb{R}^d}}
\newcommand{\Zd}{{\mathbb{Z}^d}}

\newcommand{\ep}{\varepsilon}

\renewcommand{\a}{\mathbf{a}}
\renewcommand{\b}{\mathbf{b}}

\renewcommand{\subset}{\subseteq}

\newcommand{\cu}{{\scaleobj{1.2}{\square}}}

\DeclareMathOperator{\dist}{dist}
\DeclareMathOperator*{\essinf}{ess\,inf}

\DeclareMathOperator{\diam}{diam}
\DeclareMathOperator{\supp}{supp}

\DeclareMathOperator{\size}{size}
\DeclareMathOperator{\Idd}{I}

\renewcommand{\bar}{\overline}
\renewcommand{\tilde}{\widetilde}

\newcommand{\indc}{\mathds{1}}

\newcommand{\A}{A} 

\begin{document}

\title[Boundary layers in periodic homogenization]{Quantitative analysis of boundary layers in periodic homogenization}

\begin{abstract}
We prove quantitative estimates on the rate of convergence for the oscillating Dirichlet problem in periodic homogenization of divergence-form uniformly elliptic systems. The estimates are optimal in dimensions larger than three and new in every dimension. We also prove a  regularity estimate on the homogenized boundary condition. 
\end{abstract}

\author[S. Armstrong]{Scott Armstrong}
\address[S. Armstrong]{Universit\'e Paris-Dauphine, PSL Research University, CNRS, UMR [7534], CEREMADE, 75016 Paris, France}
\email{armstrong@ceremade.dauphine.fr}

\author[T. Kuusi]{Tuomo Kuusi}
\address[T. Kuusi]{Department of Mathematics and Systems Analysis, Aalto University, Finland}
\email{tuomo.kuusi@aalto.fi}

\author[J.-C. Mourrat]{Jean-Christophe Mourrat}
\address[J.-C. Mourrat]{Ecole normale sup\'erieure de Lyon, CNRS, UMR [5669], UMPA, Lyon, France}
\email{jean-christophe.mourrat@ens-lyon.fr}

\author[C. Prange]{Christophe Prange}
\address[C. Prange]{Universit\'e de Bordeaux, CNRS, UMR [5251], IMB, Bordeaux, France}
\email{christophe.prange@math.u-bordeaux.fr}

\keywords{}
\subjclass[2010]{}
\date{\today}

\maketitle

\section{Introduction}

\subsection{Motivation and statement of results}

We consider the oscillating Dirichlet problem for uniformly elliptic systems with periodic coefficients, taking the form
\begin{equation}
\label{e.oscDP}
\left\{ 
\begin{aligned}
& -\nabla \cdot\left( \a\left(\frac x\ep \right) \nabla u^\ep(x) \right) = 0   & \mbox{in} & \ \Omega, \\
& u^\ep(x) = g\left( x, \frac x\ep \right)  & \mbox{on} & \ \partial \Omega. 
\end{aligned}
\right.
\end{equation}
Here $\ep>0$ is a small parameter, the dimension $d\geq 2$ and  
\begin{equation}
\label{e.Omega}
\Omega\subseteq\Rd \quad \mbox{is a smooth, bounded, uniformly convex domain.}
\end{equation}
The coefficients are given by a tensor $\a=\left( \a^{\alpha\beta}_{ij}\right)_{i,j=1,\ldots,L}^{\alpha,\beta=1,\ldots,d}$ and the unknown function $u^\ep = (u^{\ep}_j)_{j=1,\ldots,L}$ takes values in $\R^L$, so that the system in~\eqref{e.oscDP} can be written in coordinates as
\begin{equation*}
-\sum_{j=1}^L\sum_{\alpha,\beta=1}^d \partial_\beta \left( \a^{\alpha\beta}_{ij} \left( \frac \cdot\ep \right) \partial_\alpha u^{\ep}_j \right) = 0 \quad \mbox{in} \ \Omega, \quad \forall i \in \{ 1,\ldots,L \}.
\end{equation*}
The coefficients are assumed to satisfy, for some fixed constant $\lambda\in (0,1)$, the uniformly elliptic condition
\begin{equation}
\label{e.unifellip}
\lambda \left| \xi \right|^2 \leq \a^{\alpha\beta}_{ij}(y) \xi^\alpha_i \xi^\beta_j \leq \lambda^{-1}\left| \xi \right|^2 \quad \forall \xi = (\xi^\alpha_i) \in \R^d \times \R^L, \, y\in\Rd. 
\end{equation}
Both $\a(\cdot)$ and the Dirichlet boundary condition $g:\partial \Omega\times\Rd \to \R$ are assumed to be smooth functions,
\begin{equation}
\label{e.smooth}
\a \in C^\infty\left( \Rd ; \R^{L\times L\times d \times d} \right)
 \quad \mbox{and} \quad 
g\in C^\infty(\partial\Omega \times \Rd)
\end{equation}
and periodic in the fast variable, that is, 
\begin{equation}
\label{e.periodic}
\a(y) = \a(y+\xi) \quad \mbox{and} \quad g(x,y) = g(x,y+\xi) \quad \forall x\in \partial \Omega, \, y \in\Rd, \, \xi \in \Zd. 
\end{equation}
The goal is to understand the asymptotic behavior of the system~\eqref{e.oscDP} as $\ep \to 0$. 

\smallskip

The problem arises naturally in the theory of elliptic homogenization when one attempts to obtain a two-scale expansion of solutions of the Dirichlet problem (with non-oscillating boundary condition) near the boundary, since the oscillating term in the two scale expansion induces a locally periodic perturbation of the boundary condition of order~$O(\ep)$, cf.~\cite{BLP,GVM2}. In other words, when examining the fine structure of solutions of the Dirichlet problem with oscillating coefficients, one expects to find a boundary layer in which the solutions behave qualitatively differently than they do in the interior of the domain (for which we have a complete understanding); and the study of this boundary layer can be reduced to a problem of the form~\eqref{e.oscDP}. Unfortunately, the size and characteristics of this boundary layer as well as the behavior of the solutions in it is not well-understood, due to difficulties which arise in the analysis of~\eqref{e.oscDP}. 

\smallskip

The first asymptotic convergence result for the homogenization of the system~\eqref{e.oscDP} in general uniformly convex domains was obtained by G\'erard-Varet and Masmoudi~\cite{GVM1,GVM2}. Under the same assumptions as above, they proved the existence of an homogenized boundary condition
\begin{equation*}
\bar{g}\in L^\infty(\partial\Omega)
\end{equation*}
such that, for each $\delta>0$ and $q\in [2,\infty)$,
\begin{equation}
\label{e.GVMrate}
\left\| u^\ep - \bar{u} \right\|_{L^q(\Omega)}^q \leq C\ep^{\frac{2(d-1)}{3d+5}-\delta},
\end{equation}
where the constant $C$ depends on $(\delta,d,L,\lambda,\Omega,g,\a)$ and~$\bar{u} = \left( \bar{u}_j \right)$ is the solution of the homogenized Dirichlet problem 
\begin{equation}
\label{e.homoDP}
\left\{ 
\begin{aligned}
& -\nabla \cdot\left( \bar{\a}\nabla \bar{u}(x) \right) = 0   & \mbox{in} & \ \Omega, \\
& \bar{u}(x) = \bar{g}(x)  & \mbox{on} & \ \partial \Omega. 
\end{aligned}
\right.
\end{equation}
and~$\bar{\a}$ is the usual homogenized tensor.\footnote{In~\cite{GVM2}, the estimate~\eqref{e.GVMrate} is stated only for $q=2$, but the statement for general~$q$ can be recovered by interpolation since $L^\infty$ bounds are available for both $u^\ep$ and $\bar{u}$.} Besides giving the quantitative rate in~\eqref{e.GVMrate}, this result was the first qualitative proof of homogenization of~\eqref{e.oscDP}.

\smallskip

The asymptotic analysis of~\eqref{e.oscDP} turns out to be more difficult than that typically encountered in the theory of periodic homogenization. It is natural to approximate $\partial \Omega$ locally by hyperplanes and thus the boundary layer by solutions of a Dirichlet problem in a half-space, and these hyperplanes destroy the periodic structure of the problem. The geometry of the domain~$\Omega$ thus enters in a nontrivial way and the local behavior of the boundary layer depends on whether or not the angle of the normal vector to~$\partial \Omega$ is non-resonant with the periodic structure of~$g(x,\cdot)$ and~$\a(\cdot)$ (i.e, the lattice~$\Zd$). In domains with a different geometry -- for example, in polygonal domains as opposed to uniformly convex domains (see~\cite{AA,MV}) -- the behavior can be completely different. This is further complicated by the strength of singularities in the boundary layer and the difficulty in obtaining any regularity of the homogenized boundary condition~$\bar{g}$, which is not known to be even continuous. 

\smallskip

The lack of a periodic structure means the problem requires a quantitative approach as opposed to the softer arguments based on compactness that are more commonly used in periodic homogenization. Such a strategy was pursued in~\cite{GVM2}, based on gluing together the solutions of half-space problems with boundary hyperplanes having Diophantine (non-resonant) slopes, and it led to the estimate~\eqref{e.GVMrate}. As pointed out by the authors of~\cite{GVM2}, the exponent in~\eqref{e.GVMrate} is not optimal and was obtained by balancing two sources of error. Roughly, if one approximates~$\partial\Omega$ by too many hyperplanes, then the constant in the Diophantine condition for some of the planes is not as good, leading to a worse estimate. If one approximates with too few planes, the error in the local approximation (caused by the difference between the local hyperplane and~$\partial \Omega$) becomes large. 

\smallskip

Given the role of the problem~\eqref{e.oscDP} in quantifying asymptotic expansions in periodic homogenization, obtaining the optimal convergence rate of $\| u^\ep - \bar{u} \|_{L^p(\Omega)}$ to zero is of fundamental importance. To make a guess for how far the upper bound for the rate in~\eqref{e.GVMrate} is from being optimal, one can compare it to the known rate in the case that $\a$ is constant-coefficient (i.e.,~$\a=\bar{\a}$). In the latter case, the recent work of Aleksanyan, Shahgholian and Sj\"olin~\cite{ASS2} gives
\begin{equation}
\label{e.ASSrate}
\left\| u^\ep - \bar{u} \right\|_{L^q(\Omega)}^q \leq C\cdot \left\{ 
\begin{aligned}
& \ep^{\frac12} & \mbox{in} & \ d=2,\\
& \ep \left| \log \ep \right| & \mbox{in} & \ d=3,\\
& \ep & \mbox{in} & \ d\geq 4.
\end{aligned}
\right.
\end{equation}
One should not expect a convergence rate better than $O\big(\ep^{\frac1q}\big)$ for $\| u^\ep - \bar{u}\|_{L^q(\Omega)}$. Indeed, observe that the difference in the boundary conditions is $O(1)$ and that we should expect this difference to persist at least in an $O(\ep)$-thick neighborhood of $\partial \Omega$. Thus the solutions will be apart by at least $O(1)$ in a set of measure at least $O(\ep)$, and this already contributes $O\big(\ep^{\frac1q}\big)$ to the $L^q$ norm of the difference. The reason that the rate is worse is low dimensions is because our estimate for the boundary layer is actually optimistic: in some places (near points of $\partial \Omega$ with good Diophantine normals) the boundary layer where $|u^\ep-\bar{u}| \gtrsim 1$ will be $O(\ep)$ thick, but in other places (near points with rational normals with small denominator relative to $\ep^{-\frac12}$) the boundary layer will actually be worse, up to $O\big(\ep^{\frac12}\big)$ thick. In small dimensions (i.e., in $d=2$ and with $d=3$ being critical) the ``bad" points actually take a relatively large proportion of the surface area of the boundary, leading to a worse error. While even the analysis in the constant-coefficient case is subtle, the case of general periodic $\a(y)$ poses much greater difficulties. 

\smallskip

The main result of this paper is the following improvement of the rate~\eqref{e.GVMrate}. In dimensions $d\geq4$, we obtain the optimal convergence rate up to an arbitrarily small loss of exponent, since it agrees with~\eqref{e.ASSrate}. 

\begin{theorem}
\label{t.main}
Assume that~\eqref{e.Omega},~\eqref{e.unifellip},~\eqref{e.smooth} and~\eqref{e.periodic} hold and let $\bar{\a}$ denote the homogenized coefficients associated to~$\a(\cdot)$ obtained in periodic homogenization. 
Then there exists a function $\bar{g} \in L^\infty(\partial\Omega) $ satisfying
\begin{equation*} \label{}
\left\{ 
\begin{aligned}
& \bar{g} \in W^{s,1}(\partial\Omega) \ \forall s < \tfrac23 & \mbox{in} & \ d=2, \\
& \nabla \bar{g} \in 
L^{\frac{2(d-1)}3,\infty}(\partial \Omega) & \mbox{in} & \ d > 2, \\ 
\end{aligned}
\right.
\end{equation*}
and, for every $q\in [2,\infty)$ and $\delta>0$, a constant $C(q,\delta, d,\lambda,\a,g,\Omega)<\infty$ such that, for every $\ep\in (0,1]$, the solutions~$u^\ep$ and $\bar{u}$ of the problems~\eqref{e.oscDP} and~\eqref{e.homoDP} satisfy the estimate
\begin{equation*} \label{}
\left\| u^\ep - \bar{u} \right\|_{L^q(\Omega)}^q 
\leq C\cdot \left\{ 
 \begin{aligned}
  &  \ep^{\frac13 -\delta} & \mbox{in} & \ d= 2, \\
 &  \ep^{ \frac23-\delta } & \mbox{in} & \ d= 3, \\
 &  \ep^{1-\delta} & \mbox{in} & \ d\geq 4.
   \end{aligned} 
 \right.
\end{equation*}
\end{theorem}

\smallskip

The difference in small dimensions from our rate and~\eqref{e.ASSrate} is due to an error which arises only in the case of operators with oscillating coefficients: the largest source of error comes from the possible irregularity of the homogenized boundary condition $\bar{g}$. 
Reducing this source of error requires to improve the regularity of~$\bar{g}$. The statement asserting that $\nabla \bar{g} \in L^{\frac{2(d-1)}{3},\infty}$ in $d>2$ and $\bar{g} \in W^{\frac23-,1}$ in $d=2$ is, to our knowledge, new and the best available regularity for the homogenized boundary condition (although see Remark~\ref{r.shen} below). It is an improvement of the one proved in~\cite{GVM2}, where it was shown\footnote{This estimate was not stated in~\cite{GVM2}, but it follows from their Corollary~2.9.} that~$\nabla \bar{g} \in L^{\frac{(d-1)}{2},\infty}$ in $d>2$ and $\bar{g} \in W^{\frac12-,1}$ in $d=2$.


\begin{remark}[Optimal estimates in dimensions $d=2,3$]
\label{r.shen}
Several months after an earlier version of this paper first appeared on the arXiv, Zhongwei Shen kindly pointed out to us that our method leads to optimal estimates for the boundary layer in dimensions $d=2,3$ (up to an arbitrarily small loss of exponent). Indeed, in a very recent preprint, Shen and Zhuge~\cite{ShenZhuge_arXiv16} were able to upgrade the  regularity statement for the homogenized boundary data in Theorem~\ref{t.main}, reaching $\nabla\bar{g}\in L^q$ for any $q<d-1$ in dimension $d\geq 3$, and $\bar{g}\in W^{s,1}$ for any $s<1$ in dimension $d=2$. As stated in~\cite{ShenZhuge_arXiv16}, this regularity is expected to be optimal. Their proof of the regularity of the homogenized boundary data follows ours, with a new ingredient, namely a weighted estimate for the boundary layer. We will mention below where this new idea makes it possible to improve on our result. Using this improvement of regularity and then following our argument for estimating boundary layers leads to the following improvement of the estimates of Theorem~\ref{t.main} in $d=2,3$, which is also proved in~\cite{ShenZhuge_arXiv16}: \emph{for every $q\in [2,\infty)$ and $\delta>0$, there is a constant $C(q,\delta, d,\lambda,\a,g,\Omega)<\infty$ such that, for every $\ep\in (0,1]$, the solutions~$u^\ep$ and $\bar{u}$ of the problems~\eqref{e.oscDP} and~\eqref{e.homoDP} satisfy the estimate}
\begin{equation} \label{e.optrateslowd}
\left\| u^\ep - \bar{u} \right\|_{L^q(\Omega)}^q 
\leq C\cdot \left\{ 
 \begin{aligned}
  &  \ep^{\frac12 -\delta} & \mbox{in} & \ d= 2, \\
 &  \ep^{ 1-\delta } & \mbox{in} & \ d= 3.
   \end{aligned} 
 \right.
\end{equation}
This is optimal since it agrees with~\eqref{e.ASSrate}, up to an arbitrary loss of exponent. 
\end{remark}

\smallskip

We do not expect it to be possible to eliminate the small loss of exponent represented by~$\delta>0$ without upgrading the qualitative regularity assumption~\eqref{e.smooth} on the smoothness of~$\a$ and~$g$ to a quantitative one (for example, that these functions are analytic). Note that this regularity assumption plays an important role in the proof of Theorem~\ref{t.main} and is not a mere technical assumption or one used to control the small scales of the solutions. Rather, it is used to obtain control over the large scale behavior of the solutions via the quasiperiodic structure of the problem since it gives us a quantitative version of the ergodic theorem (see Proposition~\ref{p.ergodicthm}). In other words, the norms of high derivatives of~$\a$ and~$g$ control the ergodicity of the problem and thus the rate of homogenization. 

\smallskip

In the course of proving Theorem~\ref{t.main}, we give a new expression for the homogenized boundary condition which makes it clear that $\bar{g}(x)$ is a local, weighted average of $g(x,\cdot)$ which depends also on the normal vector $n(x)$ to $\partial \Omega$ at $x$. See~\eqref{e.gbarform1} below and~\eqref{e.defbarg} for the more precise formula.   

\smallskip

The proof of Theorem~\ref{t.main} blends techniques from previous works on the problem~\cite{GVM1,GVM2,ASS1,ASS2} with some original estimates and then combines them using a new strategy. Like the approach of~\cite{GVM2}, we cut the boundary of $\partial \Omega$ into pieces and approximate each piece by a hyperplane. However, rather than gluing approximations of the solution together, we approximate, for a fixed $x_0$, the contribution of each piece of the boundary in the Poisson formula
\begin{equation}
\label{e.formPoisson}
u^\ep(x_0) = \int_{\partial \Omega} P^\ep_\Omega(x_0,x)g\left( x,\frac x\ep \right)\,d\H^{d-1}(x).
\end{equation}
Thus, at least in the use of the Poisson formula, our approach bears a similarity to the one of~\cite{ASS1,ASS2}. 

\smallskip

The first step in the argument is to replace the Poisson kernel $P^\ep_\Omega(x_0,x)$ for the heterogeneous operator $-\nabla \cdot  \a\left( \frac\cdot\ep\right)\nabla $ by its two scale expansion, using a result of Kenig, Lin and Shen~\cite{KLS3} (based on the classical regularity theory of Avellaneda and Lin~\cite{AL1,AL2}), which states that
\begin{equation*} \label{}
P^\ep_\Omega(x_0,x) =  \overline{P}_\Omega(x_0,x) \omega^\ep(x)
\ + \ \mbox{small error,}
\end{equation*}
where $\overline{P}_\Omega(x_0,x)$ is the Poisson kernel for the homogenized operator $-\nabla \cdot \bar{\a}\nabla$ and $\omega^\ep(x)$ is a highly oscillating function which is given explicitly in~\cite{KLS3} and which depends mostly on the coefficients in an $O(\ep)$-sized neighborhood of the point $x\in\partial\Omega$. We then show that this function $\omega^\ep(x)$ can be approximated by the restriction of a \emph{smooth, $\Zd$--periodic} function on $\Rd$ which depends only on the direction of the normal derivative to $\partial \Omega$ at $x$. That is, 
\begin{equation*} \label{}
\omega^\ep(x) = \tilde\omega \left( n(x),\frac x\ep \right)
 \ + \ \mbox{small error,}
\end{equation*}
for a smooth $\Zd$--periodic function $\tilde \omega(n(x),\cdot)\in C^\infty(\Rd)$, where $n(x)$ denotes the outer unit normal to $\partial \Omega$ at $x$. This is true because the boundary of $\partial \Omega$ is locally close to a hyperplane which is then invariant under $\Zd$--translations. To bound the error in this approximation we rely in a crucial way on the $C^{1,1-}$ regularity theory of Avellaneda and Lin~\cite{AL1,AL2} up to the boundary for periodic homogenization.

\smallskip

We can therefore approximate the Poisson formula~\eqref{e.formPoisson} by
\begin{equation} \label{}
\label{e.formPoisson2}
u^\ep(x_0) = \int_{\partial \Omega} \overline{P}_\Omega(x_0,x) \,\tilde{\omega} \left( n(x) ,\frac x\ep\right) g\left( x,\frac x\ep \right) \,d\H^{d-1}(x)  
 \ + \ \mbox{small error.}
\end{equation}
Finally, we cut up the boundary of $\partial \Omega$ into small pieces which are typically of size $O(\ep^{1-})$ but sometimes as large as $O\big(\ep^{\frac12-}\big)$, depending on the non-resonance quality of the local outer unit normal to $\partial \Omega$. This chopping has to be done in a careful way, which we handle by performing a Calder\'on-Zygmund-type cube decomposition. In each piece, we freeze the macroscopic variable~$x=\bar{x}$ on both $\tilde \omega$ and $g$ and approximate the boundary by a piece of a hyperplane, making another small error. The integral on the right of~\eqref{e.formPoisson2} is then replaced by a sum of integrals, each of which is a slowly varying smooth function $\overline{P}_\Omega(x_0,\cdot)$ times the restriction of a smooth, $\ep\Zd$--periodic function $\tilde \omega\left(n(\bar{x}),\frac\cdot\ep\right)g\left(\bar{x} ,\frac\cdot\ep\right)$ to a hyperplane. This is precisely the situation in which an appropriate quantitative form of the ergodic theorem for quasiperiodic functions allows us to compute the integral of each piece, up to a (very) tiny error, which turns out to be close to the integral of $\overline{P}_\Omega(x_0,\cdot)$ times $\left\langle \tilde \omega\left(n(\bar{x}),\cdot\right)g\left(\bar{x},\cdot \right) \right\rangle$, the mean of the local periodic function. Therefore we deduce that 
\begin{equation}
\label{e.gbarform1}
u^\ep(x_0) =  \int_{\partial \Omega} \overline{P}_\Omega(x_0,x) \big\langle \tilde{\omega} (n(x),\cdot ) g( x,\cdot) \big\rangle \,d\H^{d-1}(x)  
 \ + \ \mbox{small error.}
\end{equation}
The right side is now $\bar{u}(x_0)$ plus the errors, since now we can see that the homogenized boundary condition should be  defined by
\begin{equation*} \label{}
\bar{g}(x):=  \big\langle \tilde{\omega} (n(x),\cdot )g( x,\cdot) \big\rangle,\quad x\in\partial \Omega. 
\end{equation*}

\smallskip

There is an important subtlety in the final step, since the function $\bar{g}$ is not known to be very regular. This is because we do not know how to prove that $\tilde{\omega}(n,x)$ is even continuous as a function of the direction $n\in \partial B_1$. As a consequence, we have to be careful in estimating the error made in approximating the homogenized Poisson formula with the sum of the integrals over the flat pieces. This is resolved by showing that $\bar{g}$ is continuous at every $x\in \partial\Omega$ with Diophantine normal $n(x)$, with a quantitative bound for the modulus which leads to the conclusion that $\nabla \bar{g} \in L^{\frac{2(d-1)}3-}$. This estimate is a refinement of those of~\cite{GVM2} and also uses ideas from~\cite{BLtail}; see the discussion in Section~\ref{s.halfspace}. 
 
\smallskip

We conclude this section by remarking that, while many of the arguments in the proof of Theorem~\ref{t.main} are rather specific to the problem, we expect that the high-level strategy-- based on two-scale expansion of the Poisson kernel, a suitable regularity theory (like that of Avellaneda and Lin) and the careful selection of approximating half-spaces (done here using a Calder\'on-Zygmund cube decomposition of the boundary based on the local Diophantine quality)-- to be quite flexible and useful in other situations. For instance, we expect that the analogous problem for equations in nondivergence form, studied for instance by Feldman~\cite{F}, to be amenable to a similar attack.

\subsection{Notations and basic definitions}

The indices $i,\, j,\, k,\, l$ usually stand for integers ranging between $1$ and $L$, whereas the small greek letters $\alpha,\, \beta,\, \gamma$ stand for integers ranging between $1$ and $d$. The vectors $e_i\in\R^L$, for $i=1,\ldots, L$ form the canonical basis of $\R^L$. The vector $e_d=(0,\ldots, 0,1)\in\Rd$ is the $d$-th vector of the canonical basis of $\Rd$. The notation $B_{d-1}(0,r_0)\subset\R^{d-1}$ denotes the Euclidean ball of $\R^{d-1}$ centered at the origin and of radius $r_0$. For a point $z=(z',z_d)\in\Rd$, $z'\in\R^{d-1}$ is the tangential component and $z_d\in\R$ the vertical one. The gradient $\nabla'$ is the gradient with respect to the $d-1$ first variables. The notation $M_{d}(\R)$ (resp. $M_{d-1,d}(\R)$, $M_L(\R)$) denotes the set of $d\times d$ (resp. $(d-1)\times d$, $L\times L$) matrices.

\smallskip

Unless stated otherwise, $x_0$, $x$, $\bar{x}$, $z$ denote slow variables. The point $x_0$ usually denotes a point in the interior of $\Omega$, while $x$ and $\bar{x}$ are points on the boundary ($\bar{x}$ stands for a fixed reference point). The notation $y$ stands for the fast variable, $y=\frac{x}{\ep}$. The vector $n(x)\in\partial B_1$ is the unit outer normal to $\partial\Omega$ at the point $x\in\partial\Omega$. Given $n\in\partial B_1$ and $a\in\R$, the notation $D_n(a)$ stands for the half-space $\{y\cdot n>a\}$ of $\Rd$.

\smallskip

We let~$\H^s$ denote, for $s>0$, the~$s$-dimensional Hausdorff measure on~$\Rd$. For $1\leq p\leq \infty$ and $s\geq 0$, $L^p$ is the Lebesgue space of exponent $p$, $W^{s,p}$ is the Sobolev space of regularity index $s$ and $H^s=W^{s,2}$. For $1<p<\infty$, $L^{p,\infty}$ denotes the weak $L^p$ space.

\smallskip

The first-order correctors $\chi=\chi^\beta(y)\in M_L\left(\mathbb R\right)$, indexed by $\beta=1,\ldots,\, d$, are the unique solutions of the cell problem
\begin{equation*}
\left\{ 
\begin{aligned}
& -\nabla\cdot \left(\a(y)\nabla\chi^\beta(y)\right) = \partial_\alpha\a^{\alpha\beta}(y)   & \mbox{in} & \ \mathbb T^d, \\
& \int_{\mathbb T^d}\chi^\beta(y)\, dy = 0.  &  &
\end{aligned}
\right.
\end{equation*}
The constant homogenized tensor $\bar{\a}=\left( \bar{\a}^{\alpha\beta}_{ij}\right)_{i,j=1,\ldots,L}^{\alpha,\beta=1,\ldots,d}$ is given by
\begin{equation*}
\bar{\a}^{\alpha\beta}:=\int_{\mathbb T^d}\a^{\alpha\beta}(y)\, dy+\int_{\mathbb T^d}\a^{\alpha\gamma}(y)\partial_{\gamma}\chi^\beta(y)\, dy.
\end{equation*}
Starred quantities such as $\chi^*$ refer to the objects associated to the adjoint matrix $\a^*$ defined by $(\a^*)^{\alpha\beta}_{ij}=\a^{\beta\alpha}_{ji}$, for $\alpha,\, \beta=1,\ldots,\, d$ and $i,\, j=1,\ldots,\, L$.

\smallskip

We now turn to the definition of the Poisson kernel. Let $\ep>0$ be fixed. Let $G^\ep_\Omega\in M_L(\mathbb R)$ be the Green kernel associated to the domain $\Omega$ and to the operator $-\nabla\cdot \a\left(\frac\cdot\ep\right)\nabla$. For the definition, the existence and basic properties of the Green kernel, we refer to \cite{DM}. 
The Poisson kernel $P^\ep_\Omega\in M_L(\mathbb R)$ associated to the domain $\Omega$ and to the operator $-\nabla\cdot \a\left(\frac\cdot\ep\right)\nabla$ is now defined in the following way: for all $i,\, j=1,\ldots,\, L$, for all $x_0\in\Omega$, $x\in\partial\Omega$,
\begin{equation*}
P^\ep_{\Omega,ij}(x_0,x):=-n(x)\cdot\a^*_{ik}\!\left(\frac x\ep\right)\nabla G^{\ep}_{\Omega,kj}(x,x_0).
\end{equation*}
We will use many times the following uniform bound for the Poisson kernel (cf.~\cite[Theorem 3(i)]{AL1}): there exists a constant $C(d,L,\lambda,\a,\Omega)$ uniform in $\ep$, such that for $x_0\in\Omega$, for $x\in\partial\Omega$,
\begin{equation}
\label{e.PKB}
\big|P^\ep_\Omega(x_0,x)\big|\leq \frac{C\dist(x_0,\Omega)}{|x_0-x|^d}.
\end{equation}
We denote by $\overline{P}_\Omega$ the Poisson kernel for the homogenized operator $-\nabla\cdot\bar{\a}\nabla$.

\smallskip

Let us conclude this section by two remarks on the constants. In the Diophantine condition (see Definition \ref{def.dioph}), the exponent $\kappa>\frac1{d-1}$ is fixed for the whole paper. This exponent plays no role in our work except that the condition $\kappa > \frac1{d-1}$ implies~\eqref{eq.measdioph}. As usual, $c$ and $C$ denote positive constants that may vary in each occurrence. The dependence of these constants on other parameters is made precise whenever it is necessary.

\subsection{Outline of the paper}
In the next section, we present a quantitative ergodic theorem for quasiperiodic functions and discuss the Diophantine conditions. In Section~\ref{s.CZdecomposition}, we give a triadic cube decomposition of a neighborhood of the boundary~$\partial \Omega$ using a Calder\'on-Zygmund-type stopping time argument. In Section~\ref{s.halfspace}, we analyze the half-space problem and show that the homogenized boundary condition~$\bar{g}$ is continuous at points~$x\in \partial\Omega$ with Diophantine normals~$n(x)$. In Section~\ref{s.twoscale}, we show that the two-scale expansion of the Poisson kernel is, up to a small error, locally periodic. In Section~\ref{sec.outline}, we combine all the previous ingredients to obtain an estimate of the homogenization error in terms of local errors which depend on the size of the local cube in the decomposition. In the final section, we compute the $L^q$ norm of these errors to complete the proof of Theorem~\ref{t.main}.

\section{Quantitative ergodic theorem for quasiperiodic functions}
\label{s.ergodic}

The following result is a quantitative ergodic theorem for quasiperiodic functions satisfying a Diophantine condition. Its statement can be compared to that of~\cite[Proposition 2.1]{AGK}, although the argument we give here, which is Fourier analytic, is very different from the one in~\cite{AGK}. We state it in a very general form, though we apply it later for a more specific Diophantine condition.

\begin{proposition}
\label{p.ergodicthm}
Let $\Psi : \R^{d-1} \to \R$ and $K : \R^d \to \R$ be smooth functions. 
Suppose that the matrix $N$ 
is such that it satisfies, for some $A>0$ and $f : (0,\infty) \to (0,\infty)$, the Diophantine condition
\begin{equation*} 
\left|N^T \xi \right| \geq A f(|\xi|)  \quad \forall \xi \in \Z^d \setminus \{0\}.
\end{equation*}
Then we have the estimate
\begin{multline}
\label{e.ergodic}
\left| 
\int_{\R^{d-1}} \Psi(z') K\!\left( \frac{Nz'}{\eta} \right)\,dz' 
- \widehat{K}(0) \int_{\R^{d-1}} \Psi(z') \, dz' 
\right|  \\
\leq 
\left( A^{-1} \eta \right)^k 
\left( \int_{\R^{d-1}} \left| \nabla^k \Psi(z') \right|\,dz' \right)
\left( \sum_{\xi \in \Zd\setminus\{ 0 \}} 
\left| \widehat{K}(\xi) \right| \left| f(|\xi|) \right|^{-k} \right).
\end{multline}
\end{proposition}
\begin{proof}
Assume without loss of generality (by subtracting a constant from $K$) that $\widehat{K}(0)=0$. Now we Fourier expand $K$: 
\begin{align*}
\lefteqn{
\int_{\R^{d-1}} \Psi(z') K\!\left( \frac{Nz'}{\eta} \right)\,dz' 
} \qquad & \\
& = \sum_{\xi\in \Zd \setminus \{ 0 \}} \widehat{K}(\xi) \int_{\R^{d-1}} \Psi(z') \exp\left( iN^T\xi \cdot  \frac{z'}{\eta} \right) \, dz' \\
& =  - \sum_{\xi\in \Zd \setminus \{ 0 \}} \widehat{K}(\xi) \int_{\R^{d-1}} \Psi(z') \frac{i \eta N^T \xi}{|N^T \xi|^2}   \cdot \nabla\left\{\exp\left( iN^T\xi \cdot  \frac{z'}{\eta} \right)\right\} \, dz' \\
& =  \sum_{\xi\in \Zd \setminus \{ 0 \}} \frac{i \eta \widehat{K}(\xi) N^T \xi}{|N^T \xi|^2}  \cdot \int_{\R^{d-1}} \nabla \Psi(z') \exp\left( iN^T\xi \cdot  \frac{z'}{\eta} \right)  \, dz'.
\end{align*}
After iterating this $k$ times, we obtain
\begin{multline*} \label{}
\int_{\R^{d-1}} \Psi(z') K\!\left( \frac{Nz'}{\eta} \right)\,dz'  \\
 = \sum_{\xi\in \Zd \setminus \{ 0 \}} 
 \left(  \frac{i \eta}{|N^T \xi|}  \right)^k \widehat{K}(\xi) \left(\frac{N^T \xi}{|N^T \xi|} \right)^{\otimes k} \int_{\R^{d-1}} \nabla^k \Psi(z') \exp\left( iN^T\xi \cdot  \frac{z'}{\eta} \right)  \, dz'.
\end{multline*}
Applying the Diophantine condition $|N^T \xi| \geq A f(|\xi|)$, we get
\begin{equation*} \label{}
\left| \int_{\R^{d-1}} \Psi(z') K\!\left( \frac{Nz'}{\eta} \right)\,dz' \right| 
\leq 
 \sum_{\xi\in \Zd \setminus \{ 0 \}}  \eta^k A^{-k} \left| f(|\xi|) \right|^{- k} |\widehat{K}(\xi)|
 \int_{\R^{d-1}} \left| \nabla^k \Psi(z') \right| \, dz'.
\end{equation*}
This completes the proof. 
\end{proof}

We now precisely describe the Diophantine condition we will use in our applications of Proposition~\ref{p.ergodicthm}. We take a parameter $\kappa> \frac{1}{d-1}$ to be fixed for the rest of the paper.

\begin{definition}[Diophantine direction]\label{def.dioph}
We say that $n\in\partial B_1$ is \emph{Diophantine with constant $A>0$} if
\begin{equation}\label{eq-dioph}
\left|(\Idd_d - n \otimes n)\xi\right|\geq A|\xi|^{-\kappa},\qquad \forall \xi\in\Z^{d}\setminus\{0\},
\end{equation}
where $(\Idd_d - n \otimes n)\xi$ denotes the projection of $\xi$ on the hyperplane $n^\perp$. 
\end{definition}

Let $M$ be an orthogonal matrix sending $e_d$ on $n$. We can reformulate the Diophantine condition in terms of the projection on $\R^{d-1}\times\{0\}$ of the rotated lattice elements $M^T\xi$. Denoting by $N\in M_{d,d-1}(\R)$ the matrix of the first $d-1$ columns of $M$, we have
\begin{equation*}
M^T\xi=(N^T\xi)_1e_1+\ldots\ (N^T\xi)_{d-1}e_{d-1}+(\xi\cdot n)e_d.
\end{equation*}
Thus, condition \eqref{eq-dioph} is equivalent to 
\begin{equation}\label{eq-diophbis}
\left|N^T\xi\right|\geq A|\xi|^{-\kappa},\qquad \forall \xi\in\Z^{d}\setminus\{0\}.
\end{equation}

\smallskip

The constant $A$ is necessarily less than $1$. Notice that a Diophantine vector is necessarily irrational, that is, $n\notin\R\Z^d$. The value of the exponent $\kappa$ is not important and plays no role in the paper, provided it is chosen larger than $(d-1)^{-1}$. For $A>0$, denote
\begin{equation*}
\Lambda(A):=\left\{n\in\partial B_1\,:\ n \ \mbox{satisfies}\ \mbox{\eqref{eq-dioph}}\right\}.
\end{equation*}
Then the union of $\Lambda(A)$ over all $A>0$ is a set of full measure in $\partial B_1$ with respect to $\H^{d-1}$. More precisely, we have the estimate (cf.~\cite[(2.2)]{GVM2}),
\begin{equation}\label{eq.measdioph}
\H^{d-1}\left(\Lambda(A)^c\right)\leq CA^{d-1}.
\end{equation}
We now introduce the function 
\begin{equation*}
\A\,:\,  \partial B_1\longrightarrow [0,1],
\end{equation*}
defined in the following way: for $n\in\partial B_1$,
\begin{equation}\label{eq.An}
\A(n):=\sup \left\{A\geq 0\,:\ n \in \Lambda(A) \right\}.
\end{equation}
As a consequence of \eqref{eq.measdioph}, the function $\A^{-1}$ satisfies the bound
\begin{equation*}
\H^{d-1}\left( \left\{ x\in\partial B_1 \,:\, \A^{-1}(x) > t \right\} \right) \leq Ct^{1-d}. 
\end{equation*}
Thus $\A^{-1}$ belongs to the weak Lebesgue space $L^{d-1,\infty}(\partial B_1)$ 
and the previous line can be written equivalently as
\begin{equation*} \label{}
\left\| \A^{-1} \right\|_{L^{{d-1},\infty}(\partial B_1)} \leq C.
\end{equation*}
If we consider a smooth uniformly convex domain $\Omega$, then 
the mapping
\begin{equation*}
S\,:\, x\in\partial\Omega\longmapsto n(x)\in\partial B_1,
\end{equation*}
where $n(x)$ is the unit external outer to $x\in\partial\Omega$, is a diffeomorphism (since the principal curvatures are bounded from below and above). Therefore, $\A^{-1} \circ S$ belongs to $L^{d-1,\infty}(\partial\Omega)$, 
and we have the bound
\begin{equation}
\label{e.Diophest}
\left\| \A^{-1} \circ S \right\|_{L^{{d-1},\infty}(\partial\Omega)} \leq C.
\end{equation}
Henceforth, we do not distinguish between the functions~$\A$ and~$\A\circ S$ in our notation. 

\section{Triadic cube decomposition of the boundary layer}
\label{s.CZdecomposition}

In this section, we perform a Calder\'on-Zygmund-type decomposition of the domain near the boundary which, when applied to the Diophantine constant of the normal to the boundary, will help us construct the approximation of the boundary layer.

\smallskip

We begin by introducing the notation we use for triadic cubes. For $n\in\Z$, we denote the triadic cube of size $3^n$ centered at $z\in 3^n\Zd$ by 
\begin{equation*} \label{}
\cu_n(z):= z + \left[ -\frac12 3^n, \frac12 3^n  \right)^d. 
\end{equation*}
We denote the collection of triadic cubes of size $3^n$ by
\begin{equation*} \label{}
\T_n:= \left\{ \cu_n(z)\,:\, z\in 3^n\Zd\right\}.
\end{equation*}
Notice that $\T_n$ is a partition of $\Rd$. The collection of all triadic cubes is
\begin{equation*} \label{}
\T:= \left\{ \cu_n(z) \,:\, n\in \Z, z \in 3^n\Zd \right\}
\end{equation*}
If $\cu\in\T$ has the form $\cu = \cu_n(z)$, then we denote by $\size(\cu):=3^n$ the side length of~$\cu$, the center of the cube by $\bar{x}(\cu):=z$ and, for $r>0$, we write $r\cu$ to denote the cube $z + r \cu_n$, that is, the cube centered at $z$ of side length $3^nr$. If~$\cu,\cu'\in\T$, then we say that $\cu$ is the \emph{predecessor} of~$\cu'$ if $\cu'\subseteq \cu$ and $\size(\cu) = 3 \size(\cu')$. We also say that~$\cu'$ is a \emph{successor} of $\cu$ if $\cu$ is the predecessor of $\cu'$. 

\begin{proposition}
\label{p.CZdecomposition}
Assume that $\Omega \subseteq \Rd$ is a bounded Lipschitz domain, $\delta >0$, and let
\begin{equation*}
F : \partial \Omega \to [\delta,\infty)
\end{equation*}
be a Borel measurable function. Then there exists a collection $\P \subseteq \T$ of disjoint triadic cubes satisfying the following properties:
\smallskip
\begin{enumerate}

\item[(i)] $\displaystyle\partial \Omega \subseteq \bigcup \P$.

\smallskip

\item[(ii)] For every $\cu\in \P$,
\begin{equation*} \label{}
\cu \cap \partial \Omega \neq \emptyset. 
\end{equation*}

\smallskip

\item[(iii)] For every $\cu\in\P$, 
\begin{equation*} \label{}
\essinf_{3\cu \cap \partial \Omega} F 
\leq \size(\cu).
\end{equation*}

\smallskip

\item[(iv)] There exists a positive constant $C(d,\Omega)<\infty$ such that for every $n \in \Z$,
\begin{equation*} \label{}
\# \!\left\{ \cu \in \P\,:\, \size(\cu) \geq 3^n \right\}
\leq C 3^{-n(d-1)} \H^{d-1}\left( \left\{ x \in \partial \Omega \,:\, F(x) \geq 3^{n-2} \right\} \right)
\end{equation*}

\smallskip

\item[(v)] If $\cu,\cu'\in \P$ are such that $\dist(\cu,\cu') = 0$, then 
\begin{equation*} \label{}
\frac13 \leq \frac{\size(\cu)}{\size(\cu')} \leq 3. 
\end{equation*}

\end{enumerate}
\end{proposition}

\begin{proof}
We proceed by a stopping time argument. We initialize the induction by taking $n_0 \in \Z$ large enough that $\Omega \subseteq \cu_0$ for some $\cu_0\in \T_{n_0}$ and $\essinf_{\cu_0} F \leq \size(\cu_0)$. We iteratively define a sequence $\{ \mathcal{Q}_k \}_{k\in\N}$ of subsets $\mathcal{Q}_k \subseteq \T_{n_0-k}$ satisfying, for every $\cu\in \mathcal{Q}_k$,
\begin{equation*}
\cu \cap \partial \Omega \neq \emptyset
\end{equation*}
in the following way. We take $\mathcal{Q}_0:=\{\cu_0\}$. If 
\begin{equation*} \label{}
\essinf_{\partial \Omega} F > \frac13 \size(\cu_0)
\end{equation*}
then we say that $\cu_0$ is a bad cube and we stop the procedure and set $\P:= \mathcal{Q}_0$ and $\mathcal{B}_0:=\mathcal{Q}_0$. Otherwise we set $\mathcal{G}_0=\mathcal{Q}_0$, $\mathcal B_0 = \emptyset$ and continue. Having chosen $\mathcal{Q}_0,\ldots,\mathcal{Q}_{k-1}$, and having split each of $\mathcal{Q}_j$ for $j\in\{0,\ldots,k-2\}$ into \emph{good cubes} $\mathcal{G}_j$ and \emph{bad cubes} $\mathcal{B}_j$ so that $\mathcal{Q}_j = \mathcal{G}_j\cup\mathcal{B}_j$ and $\mathcal{G}_j\cap\mathcal{B}_j=\emptyset$, we split $\mathcal{Q}_{k-1}$ into good cubes $\mathcal{G}_{k-1}$ and bad cubes $\mathcal{B}_{k-1}$ and define $\mathcal{Q}_k$ as follows. We take $\mathcal{G}_{k-1}$ to be the elements $\cu \in \mathcal{Q}_{k-1}$ satisfying both
\begin{equation}
\label{e.badnessg}
 \cu' \in \T_{n_0-k}, \ \ \cu'\cap \partial \Omega \neq \emptyset \ \  \mbox{and} \ \ \cu'\subseteq \cu 
 \implies 
 \essinf_{3\cu' \cap \partial \Omega} F \leq \frac13 \size(\cu)
\end{equation}
and
\begin{equation*}
 \cu' \in \mathcal{B}_{k-2} \ \ \implies \ \  \dist\left(\cu ,\cu'\right) > 0. 
\end{equation*}
The set of bad cubes~$\mathcal{B}_{k-1}$ is defined to be~$\mathcal{Q}_{k-1} \setminus \mathcal{G}_{k-1}$. We then define~$\mathcal{Q}_k$ to be the subcollection of~$\T_{n_0-k}$ consisting of those cubes which have nonempty intersection with~$\partial \Omega$ and are subcubes of some element of~$\mathcal{G}_{k-1}$. We stop the procedure at any point if the set of good cubes is empty. This will halt for before a finite~$k(\delta)\in\N$ owing to the assumption that~$F\geq \delta$. 

\smallskip

We now define $\P$ to be the collection of all bad cubes which intersect~$\partial \Omega$:
\begin{equation*}
\P:= \bigcup_{k\in\N} \left\{ \cu \in \mathcal{B}_{k} \,:\, \cu\cap \partial \Omega \neq \emptyset \right\}.
\end{equation*}
It is immediate from the construction that~$\P$ satisfies properties~(i), (ii), (iii) and~(v) in the statement of the proposition. We therefore have left to show only~(iv).

\smallskip

To prove~(iv), we collect all the elements of $\P$ of size $3^{n_0-(k-1)}$, for some $k\in\N$, which also satisfy~\eqref{e.badnessg}. Call this~$\P_1$, and set $\P_2 := \P \setminus \P_1$. Thus every cube of $\P_2 $ has a successor $\cu' \in \T$ such that $\cu' \cap \partial \Omega \neq \emptyset$ and  
\begin{equation*} 
F \geq \frac13 \size(\cu) \quad \mbox{a.e. in} \ 3\cu' \cap \partial \Omega. 
\end{equation*}
In particular, there is a small universal positive constant $c\in \left(0,\frac12\right]$ depending only on $d$ and the Lipschitz constant of $\Omega$ such that
\begin{equation} \label{e.cubes in P_2}
\forall \cu \in \P_2, 
\quad \H^{d-1}\left( \left\{ x \in \frac53\cu\cap \partial \Omega \,: \, F(x) \geq \frac13 \size(\cu) \right\} \right) \geq c \size(\cu)^{d-1}.
\end{equation}
Observe that property (v) ensures that the cube $\frac 53\cu\cap \partial \Omega$ is a subset of the union of the neighboring elements of $\P$ to $\cu$, that is, 
\begin{equation*}
\frac 53\cu\cap \partial \Omega 
\subseteq 
\bigcup\left\{ \cu' \in \P \,:\, \dist(\cu,\cu') = 0 \right\}.
\end{equation*}
We write $\cu\sim\cu'$ if $\cu,\cu'\in\P$ are neighboring cubes, in other words, if $\cu\neq\cu'$ and $\dist(\cu,\cu')=0$. Property (v) also ensures that every element of~$\P$ has at most~$3^{d-1}\cdot 2^d \leq C$ neighboring elements of~$\P$. 
This implies that 
\begin{align}
\label{e.blagen}
\lefteqn{
\sum_{k=0}^{n_0-n} \sum_{\cu\in \P_2\cap \mathcal{B}_k} \H^{d-1}\left( \left\{ x \in \frac53\cu\cap \partial \Omega \,: \, F(x) \geq \frac13 \size(\cu) \right\} \right)
} \qquad &  \\
& \leq \sum_{k=0}^{n_0-n} \  \sum_{\cu \in \P_2 \cap \mathcal{B}_k}  \ \sum_{\cu'\in\P,\, \cu\sim\cu'}  \H^{d-1}\left( \left\{ x \in \cu' \cap \partial \Omega \,: \, F(x) \geq \frac13 \size(\cu) \right\} \right)  \notag\\
& \leq C \sum_{\cu\in \P}  \H^{d-1}\left( \left\{ x \in \cu \cap \partial \Omega \,: \, F(x) \geq \frac19 3^n \right\} \right) \notag \\
& \leq C\H^{d-1}\left( \left\{ x\in \partial \Omega  \,:\,  F(x) \geq 3^{n-2}\right\} \right) \notag.
\end{align}
Next, for every $\cu \in \P_2$, let $\P_1(\cu)$ be the collection of elements of $\P_1$ which are subsets of $3\cu$. Observe that
\begin{equation}
\label{e.chainmunch}
\bigcup \P_1 \subset \bigcup_{\cu \in \P_2} 3\cu.
\end{equation}
Indeed, each cube $\cu$ in $\P_1$ is the neighbor of some cube~$\cu'\in\P$ with $\size(\cu') = 3\size(\cu)$. If $\cu'\not\in\P_2$, then it is also the neighbor of a cube in $\P$ that is three times larger. We continue this process, finding a chain of larger and larger cubes until we reach a cube $\cu''\in \P_2$ which is guaranteed to occur by construction. It is easy to check that $3\cu''$ contains the entire chain of cubes starting from $\cu$. This argument yields~\eqref{e.chainmunch}. 

\smallskip

Now, since $\Omega$ is a Lipschitz domain, we have that, for every $\cu'\in\P_2$,
\begin{equation}
\label{e.countings}
\# \{ \cu \in \P_1(\cu') \, : \, \size(\cu)  \geq 3^n \} \leq C  3^{-n(d-1)}\size(\cu')^{d-1}.
\end{equation}
Indeed, the $\H^{d-1}$ measure of $\partial \Omega \cap 3\cu'$ is at most $C\size(\cu')^{d-1}$ and therefore there exist at most $C\size(\cu')^{d-1} 3^{-n(d-1)}$ triadic cubes of size larger than $3^n$ which intersect it. 

\smallskip

We deduce from~\eqref{e.cubes in P_2},~\eqref{e.blagen},~\eqref{e.chainmunch} and~\eqref{e.countings} that
\begin{align*} 
 \lefteqn{\sum_{k=0}^{n_0-n} \# \bigcup_{\cu' \in \P_2 \cap \mathcal{B}_k}\{ \cu \in \P_1(\cu') \, : \, \size(\cu)  \geq 3^n \}
} \qquad \\ 
&  \leq C\sum_{k=0}^{n_0-n} 3^{(k-n)(d-1)} \# \left( \P_2 \cap \mathcal{B}_k  \right) \\ 
& \leq C 3^{-n (d-1)} \sum_{k=0}^{n_0-n} \sum_{\cu \in \P_2 \cap \mathcal{B}_k}   \H^{d-1}\left( \left\{ x \in \frac53\cu \cap \partial \Omega \,: \, F(x) \geq \frac13 \size(\cu) \right\} \right)  \\
& \leq C 3^{-n (d-1)} \H^{d-1}\left( \left\{x\in \partial \Omega\,:\,  F(x) \geq 3^{n-2}\right\} \right) 
\end{align*}
The statement (iv) follows from this. 
\end{proof}

We next construct a partition of unity of~ $\partial \Omega$ subordinate to the partition~$\left\{ \partial \Omega \cap \cu \,:\, \cu\in \P \right\}$ of~$\partial \Omega$ consisting of functions whose derivatives scale according to the size of each cube. The need to construct such a partition is the reason for requiring neighboring cubes of $\P$ to have comparable sizes in the stopping time argument, cf. property~(v) in the statement of Proposition~\ref{p.CZdecomposition}.

\begin{corollary}
\label{c.partitionofunity}
Assume the hypotheses of Proposition~\ref{p.CZdecomposition} and let $\P\subseteq\T$ be as in the conclusion. 
Then there exist a family $\left\{ \psi_\cu \in C^\infty(\Rd)\,:\, \cu\in\P\right\}$ of smooth functions and, for every $k \in \N$, there is $0<C(k,d,\Omega)<\infty$ satisfying
\begin{equation} 
\label{e.partitionofunity}
\left\{
\begin{aligned}
& 0\leq \psi_\cu \leq 1, \\  
& \supp\left(\psi_{\cu}\right) \subseteq \frac43 \cu, \\
& \sum_{\cu \in \P} \psi_\cu(x) = 1 \quad \mbox{for every} \ x\in \bigcup\P,
\quad \mbox{and} \quad \\
& \left| \nabla^k \psi_\cu \right| \leq C (\size(\cu))^{-k}.
\end{aligned}
\right.
\end{equation}
\end{corollary}
\begin{proof}
Select $\eta \in C^\infty(\Rd)$ satisfying 
\begin{equation*}
\eta\geq 0, \quad
\supp\eta \subseteq B_{\frac13}, \quad 
\eta \geq 1 \ \mbox{on} \ B_{\frac 14}, \quad
\int_{\Rd} \eta(x)\,dx = 1, \quad \mbox{and} \quad
\left| \nabla^k \eta \right| \leq C20^k.
\end{equation*}
For $r>0$, set $\eta_r(x):= r^{-d} \eta\left( \frac xr\right)$. For each $\cu\in\P$, define
\begin{equation*} \label{}
\zeta_\cu(x)
:=\indc_{\cu} \ast \eta_{\size(\cu)}(x)
= \int_{\cu} \eta_{\size(\cu)} (z-x)\,dz. 
\end{equation*}
By construction, we have that $\zeta_\cu$ satisfies, for every $k\in\N$,
\begin{equation} 
\label{e.puzeta}
 0\leq \zeta_\cu \leq 1, \ \ 
 \supp\left(\zeta_\cu\right) \subseteq \frac43 \cu, 
\ \ \mbox{and} \ \  \left| \nabla^k \zeta_\cu \right| \leq C^k (\size(\cu))^{-k}.
\end{equation}
Since $\eta \geq 1$ on $B_{\frac14}$, we have that 
\begin{equation*} \label{}
\zeta_\cu \geq c \quad \mbox{on} \ \cu. 
\end{equation*}
The previous line, the fact that $\supp(\zeta_\cu) \subseteq \frac43 \cu$ and Proposition~\ref{p.CZdecomposition}(v) imply that the function
\begin{equation*} \label{}
\zeta:= \sum_{\cu\in \P} \zeta_\cu
\end{equation*}
satisfies
\begin{equation*} \label{}
c \leq \zeta\leq C \quad \mbox{in} \ \bigcup \P. 
\end{equation*}
We also get by Proposition~\ref{p.CZdecomposition}(v) that 
\begin{equation*} \label{}
\left| \nabla^k \zeta \right| \leq C^k (\size(\cu))^{-k} \quad \mbox{in} \ \bigcup \P. 
\end{equation*}
Now define, for each $\cu\in\P$, 
\begin{equation*} \label{}
\psi_\cu:= \frac{\zeta_\cu}{\zeta}. 
\end{equation*}
It is immediate from the above construction that $\{\psi_\cu\}_{\cu\in\P}$ satisfies each of the properties in~\eqref{e.partitionofunity}. Note that the bound $\psi_\cu\leq 1$ follows from the third line of~\eqref{e.partitionofunity} and $\psi_\cu\geq 0$. 
This completes the argument. 
\end{proof}

\section{Half-space boundary layer problem}
\label{s.halfspace}

The analysis of the boundary layer in the domain $\Omega$ and the definition of the homogenized boundary condition $\bar{g}$ are based on an approximation procedure involving half-space boundary layer problems
\begin{equation}\label{eq-bdarylayerhp}
\left\{ 
\begin{aligned}
& -\nabla \cdot\left( \a(y)\nabla V\right) = 0   & \mbox{in} &\ D_{n}(a), \\
& V = V_0(y) & \mbox{on} &\ \partial D_{n}(a),
\end{aligned}
\right. 
\end{equation} 
where $n\in\partial B_1$, $a\in\R$, and $V_0$ is a $\Z^d$-periodic function. Full understanding of these boundary layers has been achieved in the works \cite{GVM1,GVM2,BLtail}. The analysis of \eqref{eq-bdarylayerhp} is very sensitive to the Diophantine properties of the normal $n$. As usual, let $M$ be an orthogonal matrix such that $Me_d=n$, and $N$ be the matrix of the $d-1$ first columns of $M$.

\smallskip

The first proposition addresses the existence and asymptotic behavior of $V$ for an arbitrary normal $n$. The derivation of the convergence away from the boundary of the half-space is based on the fact that $g$ is quasiperiodic along the boundary.

\begin{proposition}[{\cite[Theorem 1.2]{BLtail}}]
\label{eq.theobl}
For any $V_0\in C^\infty(\mathbb T^d)$, there exists a unique $C^\infty\!\left(\overline{D}_{n}(a)\right)$ solution $V$ 
to \eqref{eq-bdarylayerhp} such that
\begin{equation*}
\|\nabla V\|_{L^\infty(\{y\cdot n-t>0\})}\stackrel{t\rightarrow\infty}{\longrightarrow}0\qquad \mbox{and}\qquad \int_a^\infty\|\nabla V\cdot n\|^2_{L^\infty(\{y\cdot n-t=0\})}\, dt<\infty.
\end{equation*}
Moreover, there exists a boundary layer tail $V^\infty\in\R^L$ such that
\begin{equation}\label{eq-CV}
V(y)\stackrel{y\cdot n\rightarrow\infty}{\longrightarrow}V^\infty.
\end{equation}
When $n\notin\R\Zd$, $V^\infty$ is independent of $a$.
\end{proposition}

Some examples show that in general the convergence in \eqref{eq-CV} can be arbitrarily slow. When the normal, in addition, satisfies the Diophantine condition~\eqref{eq-dioph}, one can prove a rate of convergence.

\smallskip

We define $\mathcal V=\mathcal V(\theta,t)$ as the solution of
\begin{equation}\label{eq-bdarylayertorus}
\left\{ 
\begin{aligned}
& -\left(\begin{array}{c}N^T\nabla_\theta\\ \partial_t\end{array}\right)\cdot \left\{\b(\theta+tn)\left(\begin{array}{c}N^T\nabla_\theta\\ \partial_t\end{array}\right) \mathcal V \right\}= 0,   & \theta\in\mathbb T^d,\, t>a, \\
& \mathcal V = \mathcal V_0(\theta), & \theta\in\mathbb T^d,\, t=a.
\end{aligned}
\right. 
\end{equation} 
Here $\b$ is the coefficient matrix defined by $\b=M^T\a(\cdot)M$. Existence and uniqueness properties of $\mathcal V$ and asymptotic behavior when $t\rightarrow\infty$ are given below in Proposition \ref{p.halfspace}. Notice that if $\mathcal V$ is a solution of \eqref{eq-bdarylayertorus}, then $V$ defined by $V(Mz)=\mathcal V(Nz',z_d)$ is a solution to \eqref{eq-bdarylayerhp}.

\begin{proposition}[{\cite[Proposition 2.6]{GVM2}}]
\label{p.halfspace}
Fix $n\in\partial B_1$. For any $\mathcal V_0\in C^\infty(\mathbb T^d)$, there exists a unique solution $\mathcal V\in C^\infty(\mathbb T^d\times [a,\infty))$ such that for all $\alpha\in\N^d$, $k\in\N$, there exists a positive constant $C(d,L,\lambda,\alpha,k,\a,\mathcal V_0)<\infty$,
\begin{equation*}
\int_{\mathbb T^d}\int_a^\infty|\partial_\theta^\alpha\partial_t^{k}N^T\nabla_\theta\mathcal V|^2+|\partial_\theta^\alpha\partial_t^{k+1}\mathcal V|^2\,d\theta\,dt\leq C.
\end{equation*}
If~$n$ satisfies the Diophantine condition \eqref{eq-dioph} with positive constant $\A=\A(n)$, then for all $\alpha\in\N^d$, $k,\, m\in\N$, there exists a constant $C(d,L,\lambda,\alpha,k,m,\a,\mathcal V_0,\kappa)<\infty$ such that for all $\theta\in\mathbb T^d$, for all $T>a$,
\begin{equation}\label{eq-decaydioph}
\int_{\mathbb T^d}\int_T^\infty|\partial_\theta^\alpha\partial_t^kN^T\nabla_\theta\mathcal V|^2+|\partial_\theta^\alpha\partial_t^{k+1}\mathcal V|^2\,d\theta\, dt\leq \frac{C}{1+\A^m|T-a|^m}.
\end{equation}
\end{proposition}
For a proof see \cite[Proposition 2.6, pages 149-152]{GVM2}. Observe that~\eqref{eq-decaydioph} gives in particular for all $\alpha\in\N^d$, $k,\, m\in\N$, there exists a positive constant $C(d,L,\lambda,\alpha,k,m,\a,\mathcal V_0,\kappa)<\infty$ such that for all $\theta\in\mathbb T^d$, for all $t>a$,
\begin{equation}\label{eq-decaydiophLinfty}
\left|\partial_\theta^\alpha\partial_t^kN^T\nabla_\theta\mathcal V(\theta,t)\right|+\left|\partial_\theta^\alpha\partial_t^{k+1}\mathcal V(\theta,t)\right|\leq \frac{C}{1+\A^m|t-a|^m}.
\end{equation}
This simply follows from Sobolev's embedding theorem. Moreover, for all $\alpha\in\N^d$, $|\alpha|\geq 1$, for all $k,\, m\in\N$, for all $\theta\in\mathbb T^d$, for all $t>a$,
\begin{equation}\label{eq-decaydiophLinftynoderiv}
\left|\partial_\theta^\alpha\partial_t^k\mathcal V(\theta,t)\right|\leq \frac{C}{\A\left(1+\A^m|t-a|^m\right)},
\end{equation}
with $C(d,L,\lambda,\alpha,k,m,\a,\mathcal V_0,\kappa)<\infty$.

\smallskip

We aim now at investigating the dependence of $\mathcal V$ in terms of the normal $n$.
Our estimate below will be used to approximate the homogenized boundary data by piecewise constant data coming from the computation of boundary layers in half-spaces with good Diophantine properties. A Lipschitz estimate for the boundary layer tails appeared in~\cite[Corollary 2.9]{GVM2}, but under the assumption that both $n_1$ and $n_2$ are Diophantine normals with the same constant~$\A$ in~\eqref{eq-dioph}. Here we focus on the continuity of~$\mathcal V$ with respect to $n$. Our goal is now to prove a series of lemmas, which are tools to prove the regularity result for $\bar{g}$ stated in Theorem \ref{t.main}. These lemmas will be used in section \ref{sec.outline}. We only assume that~$n_2$ satisfies the Diophantine condition~\eqref{eq-dioph}. The argument follows that of~\cite{GVM2}, but we give full details here for the sake of completeness.

\smallskip

Let $n_1, n_2\in\partial B_1$ be two unit vectors. Assume that $n_2$ is Diophantine in the sense of \eqref{eq-dioph} with constant $\A=\A(n_2)\in (0,1]$. Let $M_1$, $M_2$ be two orthogonal matrices such that $M_1e_d=n_1$ and $M_2e_d=n_2$. We denote by $N_1$ and $N_2$ the matrices of the $d-1$ first columns of $M_1$ and $M_2$. The functions $\mathcal V_1=\mathcal V_1(\theta,t)$ and $\mathcal V_2=\mathcal V_2(\theta,t)$ are the unique solutions of \eqref{eq-bdarylayertorus} with $N$ replaced respectively by $N_1$ or $N_2$, and $\b$ replaced respectively by $\b_1$ and $\b_2$.

\smallskip

Now, the difference $\mathcal V:=\mathcal V_1-\mathcal V_2$ solves
\begin{equation}\label{eq-bdarylayertorusdiff}
\left\{ 
\begin{aligned}
& -\left(\begin{array}{c}N_1^T\nabla_\theta\\ \partial_t\end{array}\right) \cdot \left\{\b_1(\theta+tn_1)\left(\begin{array}{c}N_1^T\nabla_\theta\\ \partial_t\end{array}\right) \mathcal V\right\} =F,   & t>a, \\
& \mathcal V = 0, & t=a.
\end{aligned}
\right. 
\end{equation} 
The right-hand side is
\begin{align*}
F & = \left(\begin{array}{c}N_1^T\nabla_\theta\\ \partial_t\end{array}\right) \cdot \left\{\b_1(\theta+tn_1)-\b_2(\theta+tn_2)\right\}\left(\begin{array}{c}N_1^T\nabla_\theta\\ \partial_t\end{array}\right) \mathcal V_2 \\
& \quad +\left(\begin{array}{c}N_1^T\nabla_\theta\\ \partial_t\end{array}\right) \cdot \b_2(\theta+tn_2)\left(\begin{array}{c}N_1^T\nabla_\theta\\ \partial_t\end{array}\right) \mathcal V_2 \\
& \quad  -\left(\begin{array}{c}N_2^T\nabla_\theta\\ \partial_t\end{array}\right) \cdot \b_2(\theta+tn_2)\left(\begin{array}{c}N_2^T\nabla_\theta\\ \partial_t\end{array}\right) \mathcal V_2\\
& =\left(\begin{array}{c}N_1^T\nabla_\theta\\ \partial_t\end{array}\right)\cdot G+H,
\end{align*}
where $G:=G_1+G_2+G_3$ and $H$ are defined by
\begin{equation*}
\begin{aligned}
G_1&:=\left\{\b_1^{\cdot,\leq d}(\theta+tn_1)-\b_2^{\cdot,\leq d}(\theta+tn_2)\right\}(N_1^T-N_2^T)\nabla_\theta\mathcal V_2,\\
G_2&:=\left\{\b_1^{\cdot,\leq d}(\theta+tn_1)-\b_2^{\cdot,\leq d}(\theta+tn_2)\right\}N_2^T\nabla_\theta\mathcal V_2,\\
G_3&:=\b_2^{\cdot,\leq d}(\theta+tn_2)(N_1^T-N_2^T)\nabla_\theta\mathcal V_2,\\
H&:=(N_1^T-N_2^T)\nabla_\theta\cdot\b_2^{\leq d,\cdot}(\theta+tn_2)\left(\begin{array}{c}N_2^T\nabla_\theta\\ \partial_t\end{array}\right)\mathcal V_2,
\end{aligned}
\end{equation*}
where $\b_1^{\cdot,\leq d}=\left( {\b_1}^{\alpha\beta}_{ij}\right)_{i,j=1,\ldots,L}^{\alpha=1,\ldots,d,\,\beta=1,\ldots,d-1}$ is a submatrix of $\b_1$. The definitions of the submatrices $\b_2^{\cdot,\leq d}$ and $\b_2^{\leq d,\cdot}$ are self-explanatory. Notice that the right-hand side only involves $\mathcal V_2$. Therefore, we can use the decay estimate \eqref{eq-decaydiophLinftynoderiv} involving only the Diophantine constant $\A$. We have for all $\alpha\in\N^d$, $k,\, m\in\N$, for all $\theta\in\mathbb T^d$, for all $t>a$,
\begin{equation}\label{e.estH}
\left|\partial_\theta^\alpha\partial_t^kH(\theta,t)\right|\leq \frac{C|n_1-n_2|}{1+\A^m|t-a|^m}
\end{equation}
and the following estimates for $G_1,\, G_2$ and $G_3$,
\begin{align}
&\left|\partial_\theta^\alpha\partial_t^kG_1(\theta,t)\right|\leq \frac{C|n_1-n_2|^2|t-a|}{\A(1+\A^m|t-a|^m)},\label{e.estG_1}\\
&\left|\partial_\theta^\alpha\partial_t^kG_2(\theta,t)\right|\leq \frac{C|n_1-n_2||t-a|}{1+\A^m|t-a|^m},\label{e.estG_2}\\
&\left|\partial_\theta^\alpha\partial_t^kG_3(\theta,t)\right|\leq \frac{C|n_1-n_2|}{\A(1+\A^m|t-a|^m)},\label{e.estG_3}
\end{align}
with $C(d,L,\lambda,\alpha,k,m,\a,\mathcal V_0,\kappa)<\infty$.

\begin{lemma}
\label{lem.estV}
Let $\mathcal V$ be the solution of~\eqref{eq-bdarylayertorusdiff}.
For every $s\in\N$, there exists a constant $C(d,L,\lambda,s,\a)<\infty$ such that
\begin{multline}\label{e.estV}
\int_a^\infty\|N_1^T\nabla_\theta\mathcal V\|_{H^s(\mathbb T^d)}^2+\|\partial_t\mathcal V\|_{H^s(\mathbb T^d)}^2\,dt\\
\leq C\left(\int_a^\infty\|(t-a)H\|_{H^s(\mathbb T^d)}^2+\|G\|_{H^s(\mathbb T^d)}^2\,dt\right).
\end{multline}
\end{lemma}
\begin{proof}
The proof is by induction on the number of derivatives $s$. The result follows from simple energy estimates carried out on system \eqref{eq-bdarylayertorusdiff}.

\smallskip

\emph{Step 1.} In this first step, we prove \eqref{e.estV} for $s=0$. Testing against $\mathcal V$ and integrating by parts, we get
\begin{multline}\label{e.ipps=0}
\lambda\left\|\left(\begin{array}{c}N_1^T\nabla_\theta\\ \partial_t\end{array}\right)\mathcal V\right\|_{L^2(\mathbb T^d\times [a,\infty))}^2\\
\leq-\int_a^\infty\int_{\mathbb T^d}G\cdot \left(\begin{array}{c}N_1^T\nabla_\theta\\ \partial_t\end{array}\right)\mathcal V\,d\theta\,dt
+\int_a^\infty\int_{\mathbb T^d}H\mathcal V\,d\theta\,dt.
\end{multline}
We now estimate the right-hand side above. For the first term, we have
\begin{multline*}
\left|\int_a^\infty\int_{\mathbb T^d}G\cdot \left(\begin{array}{c}N_1^T\nabla_\theta\\ \partial_t\end{array}\right)\mathcal V\,d\theta\,dt\right|\\
\leq \frac{1}{2\lambda}\|G\|_{L^2(\mathbb T^d\times [a,\infty))}^2+\frac{\lambda}{2}\left\|\left(\begin{array}{c}N_1^T\nabla_\theta\\ \partial_t\end{array}\right)\mathcal V\right\|_{L^2(\mathbb T^d\times [a,\infty))}^2,
\end{multline*}
so that we can easily swallow the second term in the left-hand side of \eqref{e.ipps=0}. For the second term on the right-hand side of \eqref{e.ipps=0}, we use Hardy's inequality. This yields
\begin{multline*}
\left|\int_a^\infty\int_{\mathbb T^d}H\mathcal Vd\theta\,dt\right|=\left|\int_a^\infty\int_{\mathbb T^d}(t-a)H\frac{\mathcal V}{t-a}\,d\theta\,dt\right|\\\leq C\|(t-a)H\|_{L^2(\mathbb T^d\times[a,\infty))}\|\partial_t\mathcal V\|_{L^2(\mathbb T^d\times[a,\infty))}.
\end{multline*}
Young's inequality makes it now possible to reabsorb the $L^2$ norm of $\partial_t\mathcal V$ on the left-hand side of \eqref{e.ipps=0}.

\smallskip

\emph{Step 2.} We now estimate the higher-order derivatives by induction. The arguments are basically the same as for $s=1$ since the system satisfied by $\partial_\theta^\alpha\mathcal V$, for $\alpha\in\N^d$, has basically the same structure than the system \eqref{eq-bdarylayertorusdiff} for $\mathcal V$. In particular, $\partial_\theta^\alpha\mathcal V$ is zero on the boundary $\mathbb T^d\times\{a\}$. Let us do the proof only for $s=1$ as the higher-order cases are treated in the same way. Let $\alpha\in\N^d$ be such that $|\alpha|=1$. Testing the equation for $\partial_\theta^\alpha\mathcal V$ against $\partial_\theta^\alpha\mathcal V$ and integrating by parts, we get
\begin{multline}
\lambda\left\|\left(\begin{array}{c}N_1^T\nabla_\theta\\ \partial_t\end{array}\right)\partial_\theta^\alpha\mathcal V\right\|_{L^2(\mathbb T^d\times [a,\infty))}^2\\
\leq-\int_a^\infty\int_{\mathbb T^d}\partial_\theta^\alpha\b_1(\theta+tn_1)\left(\begin{array}{c}N_1^T\nabla_\theta\\ \partial_t\end{array}\right)\mathcal V\cdot\left(\begin{array}{c}N_1^T\nabla_\theta\\ \partial_t\end{array}\right)\partial_\theta^\alpha\mathcal V\,d\theta\,dt\\
-\int_a^\infty\int_{\mathbb T^d}\partial_\theta^\alpha G\cdot \left(\begin{array}{c}N_1^T\nabla_\theta\\ \partial_t\end{array}\right)\partial_\theta^\alpha\mathcal V\,d\theta\,dt
+\int_a^\infty\int_{\mathbb T^d}\partial_\theta^\alpha H\partial_\theta^\alpha\mathcal V\,d\theta\,dt.
\end{multline}
We introduce the following notations
\begin{equation*}
\tilde{G}:=\partial_\theta^\alpha\b_1(\theta+tn_1)\left(\begin{array}{c}N_1^T\nabla_\theta\\ \partial_t\end{array}\right)\mathcal V+\partial_\theta^\alpha G,\qquad \tilde{H}:=\partial_\theta^\alpha H.
\end{equation*}
The proof of our estimate now follows exactly the scheme of Step 1 above.
\end{proof}

Notice that the estimates \eqref{e.estH},~\eqref{e.estG_1},~\eqref{e.estG_2} and~\eqref{e.estG_3} yield
\begin{equation}
\label{e.boundH}
\int_a^\infty\|(t-a)H\|_{H^s(\mathbb T^d)}^2\,dt\leq \frac{C|n_1-n_2|^2}{\A^3},
\end{equation}
and
\begin{equation}
\label{e.boundG}
\int_a^\infty\|G\|_{H^s(\mathbb T^d)}^2\,dt\leq \frac{C|n_1-n_2|^2}{\A^3}\left(1+\frac{|n_1-n_2|^2}{\A^2}\right).
\end{equation}

In the next lemma, we give control of higher derivatives in $t$.

\begin{lemma}\label{lem.estVt}
Let $\mathcal V$ be the solution of~\eqref{eq-bdarylayertorusdiff} and $s\in\N$. There exist constants $\nu_0(d)<\infty$ and $C(d,L,\lambda,s,\a)<\infty$ such that for any $x_1,\, x_2\in\partial\Omega$  and $n_i := n(x_i)$, $i \in \{1,2\}$, if $n_2$ is Diophantine with positive constant $\A$, and $|n_1-n_2|<\nu_0$, then
\begin{multline}\label{e.estVt}
\|N_1^T\nabla_\theta\mathcal V\|_{H^s(\mathbb T^d\times[a,\infty))}^2+\|\partial_t\mathcal V\|_{H^s(\mathbb T^d\times[a,\infty))}^2\leq \frac{C|n_1-n_2|^2}{\A^3}\left(1+\frac{|n_1-n_2|^2}{\A^2}\right).
\end{multline}
\end{lemma}

\begin{proof}
\emph{Step 1.} The derivatives in $\theta$ are handled through Lemma \ref{lem.estV}. The proof is by induction on the number of derivatives in $t$. Let us prove
\begin{equation*}
\int_a^\infty\|N_1^T\nabla_\theta\mathcal V\|_{H^s(\mathbb T^d)}^2+\|\partial_t\mathcal V\|_{H^s(\mathbb T^d)}^2\,dt\leq  \frac{C|n_1-n_2|^2}{\A^3}\left(1+\frac{|n_1-n_2|^2}{\A^2}\right).
\end{equation*}
The issue is that $\partial_t\mathcal V$ is not $0$ on the boundary, on the contrary of tangential derivatives. We therefore have to get some control on $\partial_t\mathcal V(\theta,0)$, before lifting it. We have
\begin{multline}\label{e.estpartial2V}
\partial_t^2\mathcal V=\frac{1}{\b_1^{d+1,d+1}}\Biggl(-\partial_t\left(\b_1^{d+1,d+1}(\theta+tn_1)\right)\partial_t\mathcal V-N_1^T\nabla_\theta\cdot\b_1^{\leq d,d+1}\partial_t\mathcal V\Biggr.\\
\Biggl.-\partial_t\left(\b_1^{d+1,\leq d}N_1^T\nabla_\theta\mathcal V\right)-N_1^T\nabla_\theta\cdot\b_1^{\leq d,\leq d}N_1^T\nabla_\theta\mathcal V+\left(\begin{array}{c}N_1^T\nabla_\theta\\ \partial_t\end{array}\right)\cdot G+H\Biggr),
\end{multline}
where $\b_1^{d+1,d+1}$, $\b_1^{\leq d,d+1}$, $\b_1^{d+1,\leq d}$ and $\b_1^{\leq d,\leq d}$ are submatrices of $\b_1$. Consequently, using \eqref{e.estV} to estimate the first four terms on the right-hand side above, and \eqref{e.boundH} and \eqref{e.boundG} to estimate the source terms $G$ and $H$, we get
\begin{equation*}
\int_a^\infty\|\partial_t^2\mathcal V\|_{H^s(\mathbb T^d)}^2\, dt\leq \frac{C|n_1-n_2|^2}{\A^3}\left(1+\frac{|n_1-n_2|^2}{\A^2}\right).
\end{equation*}
Therefore, $\partial_t\mathcal V(\theta,t)\eta(t)$, with $\eta\in C^\infty_c(\R)$ equal to $1$ in the neighborhood of $0$, is a lifting of $\partial_t\mathcal V(\theta,0)\in H^{\frac12}(\mathbb T^d)$. Then, $\mathcal W:=\partial_t\mathcal V-\partial_t\mathcal V\eta(t)$ solves
\begin{equation*}
\left\{ 
\begin{aligned}
& -\left(\begin{array}{c}N_1^T\nabla_\theta\\ \partial_t\end{array}\right) \cdot \left\{\b_1(\theta+tn_1)\left(\begin{array}{c}N_1^T\nabla_\theta\\ \partial_t\end{array}\right) \mathcal W\right\} =\left(\begin{array}{c}N_1^T\nabla_\theta\\ \partial_t\end{array}\right)\cdot\tilde{G}+\tilde{H},   & t>a, \\
& \mathcal W = 0, & t=a,
\end{aligned}
\right. 
\end{equation*} 
where
\begin{equation*}
\begin{aligned}
\tilde{G} & :=\b_1(\theta+tn_1)\left(\begin{array}{c}N_1^T\nabla_\theta\\ \partial_t\end{array}\right)\left(\partial_t\mathcal V\eta(t)\right)+\partial_t\left(\b_1(\theta+tn_1)\right)\left(\begin{array}{c}N_1^T\nabla_\theta\\ \partial_t\end{array}\right) \mathcal V+\partial_tG,\\
\tilde{H}&:=\partial_tH.
\end{aligned}
\end{equation*}
Integrating by parts, we get
\begin{multline}
\lambda\left\|\left(\begin{array}{c}N_1^T\nabla_\theta\\ \partial_t\end{array}\right)\partial_t\mathcal V\right\|_{L^2(\mathbb T^d\times [a,\infty))}^2\\
\leq-\int_a^\infty\int_{\mathbb T^d}\partial_\theta^\alpha\b_1(\theta+tn_1)\left(\begin{array}{c}N_1^T\nabla_\theta\\ \partial_t\end{array}\right)\mathcal V\cdot\left(\begin{array}{c}N_1^T\nabla_\theta\\ \partial_t\end{array}\right)\partial_t\mathcal V\,d\theta\,dt\\
-\int_a^\infty\int_{\mathbb T^d}\partial_t G\cdot \left(\begin{array}{c}N_1^T\nabla_\theta\\ \partial_t\end{array}\right)\partial_t\mathcal V\,d\theta\,dt
+\int_a^\infty\int_{\mathbb T^d}\partial_t H\partial_t\mathcal V\,d\theta\,dt,
\end{multline}
which is estimated exactly as in Lemma \ref{lem.estV}. 

\smallskip

\emph{Step 2.} Estimating higher-order derivatives is done in the same way. Let $k\in\mathbb N$. Assume by induction that for all $s\in\mathbb N$ there exists a constant $C(d,L,\lambda,s,k,\a)<~\infty$ such that 
\begin{multline}\label{e.estVtrec}
\|N_1^T\nabla_\theta\mathcal V\|_{H^k\left([a,\infty);H^s(\mathbb T^d)\right)}^2+\|\partial_t\mathcal V\|_{H^k\left([a,\infty);H^s(\mathbb T^d)\right)}^2\\
\leq \frac{C|n_1-n_2|^2}{\A^3}\left(1+\frac{|n_1-n_2|^2}{\A^2}\right).
\end{multline}
Our goal is now to show \eqref{e.estVtrec} for $k$ replaced by $k+1$. Differentiating the equation \eqref{eq-bdarylayertorusdiff} $k+1$ times with respect to $t$, we get
\begin{multline}\label{e.estpartialk+3V}
\partial_t^{k+3}\mathcal V=\frac{1}{\b_1^{d+1,d+1}}\Bigg(-\partial_t\left(\b_1^{d+1,d+1}(\theta+tn_1)\right)\partial_t^{k+2}\mathcal V-N_1^T\nabla_\theta\cdot\b_1^{\leq d,d+1}\partial_t^{k+2}\mathcal V\\
-\partial_t\left(\b_1^{d+1,\leq d}N_1^T\nabla_\theta\partial_t^{k+1}\mathcal V\right)-N_1^T\nabla_\theta\cdot\b_1^{\leq d,\leq d}N_1^T\nabla_\theta\partial_t^{k+1}\mathcal V\\
+\left(\begin{array}{c}N_1^T\nabla_\theta\\ \partial_t\end{array}\right)\cdot \partial_t^{k+1}G+\partial_t^{k+1}H+\left(\begin{array}{c}N_1^T\nabla_\theta\\ \partial_t\end{array}\right)\cdot\left(\mbox{Lower-order terms}\right)\Bigg),
\end{multline}
where
\begin{equation*}
\mbox{Lower-order terms}=-\sum_{l=0}^{k}\frac{(k+1)!}{l!(k+1-l)!}\partial_t^{k+1-l}\left(\b_1(\theta+tn_1)\right)\left(\begin{array}{c}N_1^T\nabla_\theta\\ \partial_t\end{array}\right)\partial_t^l\mathcal V.
\end{equation*}
Notice that the structure of \eqref{e.estpartialk+3V} is similar to the one of \eqref{e.estpartial2V}. All the terms on the right-hand side of \eqref{e.estpartialk+3V} involve at most $k+2$ derivatives in $t$ and therefore, they can be estimated using \eqref{e.estVtrec}. The rest of the proof is completely analogous to Step 1 above.
\end{proof}

\section{Two-scale expansion of the Poisson kernel}
\label{s.twoscale}

An important ingredient in our analysis is the two-scale expansion result of Kenig, Lin and Shen~\cite[Theorem 3.8]{KLS3} for the Poisson kernel~$P^\ep_\Omega$ associated to the domain $\Omega$ and to the operator $-\nabla\cdot \a\left(\frac\cdot\ep\right)\nabla$. They proved that, for every $x_0\in\Omega$ and $x\in\partial\Omega$,
\begin{equation}\label{eq:formulaKLS}
P^\ep_\Omega(x_0,x)=\overline{P}_\Omega(x_0,x)\omega^\ep(x)+R^\ep(x_0,x),
\end{equation}
where $\overline{P}_\Omega$ is the Poisson kernel associated to the domain $\Omega$ and the homogenized operator $-\nabla\cdot\overline\a\nabla$, the function $\omega^\ep$ is a highly oscillating kernel whose definition is given below and the remainder term $R^\ep$ satisfies, for a positive constant $C(d,L,\lambda,\a,\Omega)<\infty$,
\begin{equation}
\label{e.Repbound}
\left|R^\ep(x_0,x)\right|\leq C\ep\left|x_0-x \right|^{-d} \log\left(2+\frac{\left|x_0-x \right|}\ep\right).
\end{equation}

\smallskip

By ellipticity of $\overline\a$, the matrix $\overline\a n(x)\cdot n(x)\in M_L(\R)$ is invertible; we denote its inverse by $h(x)$. The oscillating part $\omega^\ep(x)$ of the kernel is then defined by for all $1\leq i,\, j\leq L$, 
\begin{equation}\label{eq:oscillatingomega}
\omega^\ep_{ij}(x):=h_{ik}(x)n(x)\cdot\nabla\Phi^{*,\ep}_{lk}(x)\cdot n(x)\a_{lj}\!\left(\frac x\ep \right)n(x)\cdot n(x),
\end{equation}
where $\Phi^{*,\ep,lk}$ is the Dirichlet corrector associated to the adjoint matrix $\a^*$,
\begin{equation}
\left\{ 
\begin{aligned}
& -\nabla \cdot\left(\a^*\left(\frac x\ep \right) \nabla \Phi^{*,\ep}\right) = 0   & \mbox{in} & \ \Omega, \\
& \Phi^{*,\ep} = \mathbf{p}(x) & \mbox{on} & \ \partial \Omega,
\end{aligned}
\right.
\end{equation}
where $\mathbf p^{\alpha}_j(x):=x_\alpha e_j$ for each $x\in\Rd$, $1\leq\alpha\leq d$ and $1\leq j\leq L$.

\smallskip

With an eye toward Proposition~\ref{p.ergodicthm}, we notice that, after zooming in at a mesoscopic scale $\ep\leq r\ll 1$ in the vicinity of one boundary point $\bar{x}$, the non-oscillating functions $h=h(x)$ and $n=n(x)$ are almost constant, equal to $h(\bar{x})$ and $n(\bar{x})$. The oscillating matrix $\a\left(\frac\cdot\ep\right)$ is quasiperiodic along the boundary of the half-space $D_{n(\bar{x})}(c(\bar{x}))$ tangent to $\partial\Omega$ at $\bar{x}$. 
However, we do not know a priori how the normal derivative $n(x)\cdot\nabla\Phi^{*,\ep}(x)$ of the Dirichlet corrector oscillates. 

\smallskip

The goal is therefore to describe the behavior of $\Phi^{*,\ep}$ close to $\bar{x}$ in terms of intrinsic (and periodic) objects, namely cell correctors and half-space boundary layer correctors. More precisely, we will prove the following expansion for $\Phi^{*,\ep}$.

\begin{proposition}\label{prop:expphieps}
For all $\rho\in(0,1)$, there exists~$C(d,L,\lambda,\a,\Omega,\rho)<\infty$ 
such that, for every $\ep\leq r\leq\ep^{\frac12}$, we have 
\begin{multline}
\label{eq:errorestpropexpphieps}
\left\|\nabla\left(\Phi^{*,\ep}(x)-\mathbf{p}(x)-\ep\chi^*\!\left(\frac x\ep \right)-\ep V^*\!\left(\frac x\ep \right)\right)\right\|_{L^\infty(\overline{\Omega}\cap B(\bar{x},r))} \\
\leq C\left(\ep^{\frac12}+\frac{r^{2+\rho}}{\ep^{1+\rho}}\right) \wedge 1,
\end{multline}
where $\chi^*$ is the cell corrector associated to $-\nabla\cdot\a^*(y)\nabla$ and $V^*=V^*(y)$ is the boundary layer corrector solving
\begin{equation}\label{eq.blcorrstar}
\left\{ 
\begin{aligned}
& -\nabla \cdot\left( \a^*(y)\nabla V^*\right) = 0,  & \mbox{in} &\ D_{n(\bar{x})}\left(\frac{c(\bar{x})}{\ep}\right), \\
& V^* = -\chi^*(y),& \mbox{on} &\ \partial D_{n(\bar{x})}\left(\frac{c(\bar{x})}{\ep}\right).
\end{aligned}
\right. 
\end{equation} 
\end{proposition}

It follows immediately from the proposition that we can approximate $\omega^\ep$.

\begin{corollary}\label{cor.approxomegaeps}
For all $\rho\in(0,1)$, there exists~$C(d,L,\lambda,\a,\Omega,\rho)<\infty$ such that, for every $\ep\leq r\leq\ep^{\frac 12}$, we have
\begin{equation}\label{eq:errorestpropexpphiepscor}
\left\|\omega^\ep(x)-\tilde{\omega}^\ep(n(\bar{x}),x)\right\|_{L^\infty(\partial{\Omega}\cap B(\bar{x},r))}\leq C\left(\ep^{\frac12}+\frac{r^{2+\rho}}{\ep^{1+\rho}}\right) \wedge 1,
\end{equation}
where we denote, for $1\leq i,\, j\leq L$,
\begin{multline}\label{e.deftildeomegaeps}
\tilde{\omega}^\ep_{ij}(n(\bar{x}),x) \\:=h_{ik}(\bar{x})n(\bar{x})\cdot\nabla\left(\mathbf{p}_{lk}(x)-\ep\chi^{*}_{lk}\!\left(\frac x\ep \right)-\ep V^{*}_{lk}\!\left(\frac x\ep \right)\right)
\cdot n(\bar{x})\a_{lj}\!\left(\frac x\ep \right)n(\bar{x})\cdot n(\bar{x}).
\end{multline}
\end{corollary}

Before going into the details of the proof of Proposition \ref{prop:expphieps}, let us comment on the boundary layer corrector $V^*$ solving \eqref{eq.blcorrstar}. The existence and uniqueness of $V^*$ is a consequence of Proposition~\ref{eq.theobl}. The boundary layer corrector is bounded in $L^\infty\left(D_{n(\bar{x})}\left(\frac{c(\bar{x})}{\ep}\right)\right)$. Below, we will need the following estimate of the derivatives of $V^*+\chi^*$ in a layer close to the boundary of the half-space.

\begin{lemma}\label{lem.estlayer}
Let $0<\mu<1$. There exists a positive constant $C(d,L,\lambda,\a,\Omega)<\infty$, such that for all $n\in\partial B_1$, for all $a\in\R$, for all solution $V^*$ to 
\begin{equation}
\left\{ 
\begin{aligned}
& -\nabla \cdot\left( \a^*(y)\nabla V^*\right) = 0,   & \mbox{in} &\ D_n(a), \\
& V^* = -\chi^*(y), & \mbox{on} &\ \partial D_n(a),
\end{aligned}
\right. 
\end{equation} 
we have the estimate
\begin{equation}\label{eq.estlayer}
\|V^*+\chi^*\|_{C^{3,\mu}(a<y\cdot n<a+1)}\leq C.
\end{equation}
\end{lemma}

\begin{proof}
Let $\bar{y}$ be a point on the hyperplane $\partial D_{n}(a)$, i.e. $\bar{y}\cdot n=a$. Estimate \eqref{eq.estlayer} follows from applying the local boundary Schauder theory  to $V^*+\chi^*$, which solves
\begin{equation*}
\left\{ 
\begin{aligned}
& -\nabla\cdot\left( \a^*(y)\nabla (V^*+\chi^*)\right) = \nabla\cdot\a^*, & B(\bar{y},2)\cap D_{n}(a), \\
& V^*+\chi^*=0, & B(\bar{y},2)\cap\partial D_{n}(a).
\end{aligned}
\right. 
\end{equation*}
Therefore, the local boundary $C^{3,\mu}$ estimate implies
\begin{equation*}
\|V^*+\chi^*\|_{C^{3,\mu}(B(\bar{y},2)\cap D_{n}(a))}\leq C\left(\|V^*+\chi^*\|_{L^{\infty}(B(\bar{y},4)\cap D_{n}(a))}+\|\a^*\|_{C^{2,\mu}(\mathbb T^d)}\right)\leq C,
\end{equation*}
which yields the result.
\end{proof}

\begin{proof}[Proof of Proposition \ref{prop:expphieps}]
Let $\ep\leq r\leq\ep^{\frac12}\ll1$. The proof relies on the uniform regularity theory developed by Avellaneda and Lin~\cite{AL1}. There is no essential difficulty, but some technical aspects have to be handled.

\smallskip

\emph{Step 1.} 
We first rotate and translate the domain in order to work in the situation where $\bar{x}=0$, $n(\bar{x})=e_d$. We ignore the effect of translation and assume right away that $c(\bar{x})=0$. Rotating the domain will change both the coefficient matrix and the boundary condition. 
Denoting as usual by $n(\bar{x})$ the outer unit normal at $\bar{x}$, we take $M\in M_d(\R)$ an orthogonal matrix sending the $d$-th vector of the canonical basis $e_d$ on $n(\bar{x})$. Let $N\in M_{d,d-1}(\R)$ be the matrix of the first $d-1$ columns of $M$ and let $n$ stand for $n(\bar{x})$. The matrix $\b$ is defined as above by $\b=M^T\a(\cdot) M$. 
The rotated domain $M^T\Omega$ is again denoted by $\Omega$. Our convention is to keep the trace of the rotation $M$ in the notations whenever it may change some calculations or estimates. Otherwise, we overlook the dependence in $M$.

\smallskip

There exists $0<\ep^{\frac12}\ll r_0\sim 1$, such that in $\left(B_{d-1}(0,r_0)\times(-r_0,r_0)\right)\cap\partial \Omega$ we can write $\partial\Omega$ as a graph $x_d:=\varphi(x')$, where $\varphi$ is (at least) a $C^{3}$ function. Notice that $\varphi(0)=0$ as well as $\nabla\varphi(0)=0$. By Taylor expanding $\varphi$ around $0$, we get for all $|x'|<r_0$,
\begin{equation}\label{eq:varphiparab}
\varphi(x')=\int_0^1\nabla^2\varphi(tx')\cdot x'^2(1-t)\,dt=:\tilde{\varphi}(x')\cdot x'^2,
\end{equation}
where $\tilde{\varphi}$ as well as its first-order derivative are bounded in $L^\infty(B(0,r_0))$.

\smallskip

Notice that the rotated Dirichlet corrector, $\Phi^{*,\ep,lk}(M\cdot)$ solves 
\begin{equation}
\left\{ 
\begin{aligned}
& -\nabla \cdot\left(\mathbf b^*\!\left(\frac{Mx}{\ep}\right) \nabla \Phi^{*,\ep}(Mx) \right)= 0   & \mbox{in} & \ \Omega, \\
& \Phi^{*,\ep}(Mx) = \mathbf{p}(Mx) & \mbox{on} & \ \partial\Omega.
\end{aligned}
\right.
\end{equation}

\smallskip

\emph{Step 2.} The proof of the error estimate \eqref{eq:errorestpropexpphieps} relies on the local boundary Lipschitz estimate uniform in $\ep$ proved in \cite{AL1}. Let us point out that this estimate holds even if the coefficients of the rotated system are quasiperiodic, regardless of whether a Diophantine condition is satisfied. Indeed, the uniform regularity theory holds for the original system, before the rotation, which has periodic coefficients.

\smallskip

Denote
\begin{equation*}
s^\ep(\cdot):=\Phi^{*,\ep}(M\cdot)-\mathbf{p}(M\cdot)-\ep\chi^*\!\left(\frac {M\cdot}\ep \right)-\ep V^*\!\left(\frac {M\cdot}\ep \right).
\end{equation*}
The error $s^\ep$ is a weak solution of
\begin{equation*}
\left\{ 
\begin{aligned}
& -\nabla \cdot\left( \mathbf b^*\left(\frac{Mx}{\ep}\right) \nabla s^\ep \right) = 0   & \mbox{in} & \ \Omega\cap B(0,2r), \\
& s^\ep = -\ep\chi^*\!\left(\frac {Mx}\ep \right)-\ep V^*\!\left(\frac {Mx}\ep \right) & \mbox{on} & \ \partial \Omega\cap B(0,2r).
\end{aligned}
\right.
\end{equation*}
The boundary Lipschitz estimate \cite[Lemma 20]{AL1} gives, for $0<\rho<1$,
\begin{multline}\label{eq:lipest}
\left\|\nabla s^\ep \right\|_{L^\infty(\overline{\Omega}\cap B(0,r))}\leq C\biggl\{\frac{1}{r}\left\| s^\ep \right\|_{L^\infty(\Omega\cap B(0,2r))}\biggr.\\
\biggl.+r^\rho\left[\ep\chi^*\!\left(\frac {Mx}\ep \right)+\ep V^*\!\left(\frac {Mx}\ep \right)\right]_{C^{1,\rho}(\partial\Omega\cap B(0,2r))}\biggr\},
\end{multline}
with $C$ depending only on $d$, $\lambda$, the H\"older semi-norm of $\a$, $\rho$ and $\Omega$ and, in particular, independent of $\ep$. We now estimate each term on the right-hand side of \eqref{eq:lipest}. For both terms we crucially need to use cancellation properties in the boundary condition.

\smallskip

\emph{Step 3.} We now concentrate on the first term on the right-hand side of \eqref{eq:lipest}. The claim is that
\begin{equation}\label{eq:boundstep3}
\left\|s^\ep\right\|_{L^\infty(\Omega\cap B(0,2r))}\leq C\big(r\ep^{\frac12}+r^2\big).
\end{equation}
We first use the Agmon-type maximum principle of \cite[Theorem 3(ii)]{AL1} to get
\begin{equation}\label{e.agmonpple}
\left\|s^\ep\right\|_{L^\infty(\Omega\cap B(0,2r))}\leq \left\|s^\ep\right\|_{L^\infty(\Omega)}\leq C\ep.
\end{equation}
Notice that this $L^\infty$ bound is nothing but a consequence of the integral representation formula for $s^\ep$ and the Poisson kernel bound~\eqref{e.PKB}: for every $x_0\in\Omega$,
\begin{align*}
\left|s^\ep(x_0)\right|
& =\ep\left|\int_{\partial\Omega} P_\Omega(x_0,x)\left(\chi^*\!\left(\frac {Mx}\ep \right)+\ep V^*\!\left(\frac {Mx}\ep \right)\right)\, d\H^{d-1}(x)\right|\\
& \leq C\ep\int_{\partial\Omega}\frac{\dist(x_0,\partial\Omega)}{|x_0-x|^d}d\H^{d-1}(x) \\
& \leq C\ep.
\end{align*}
However estimate \eqref{e.agmonpple} only yields  $O(\frac\ep r)$ on the right-hand side of \eqref{eq:errorestpropexpphieps}, which is not nearly good enough. To get a better estimate, we must use the cancellation properties of the boundary condition close to the origin. Using
\begin{equation}\label{eq:cancel}
\ep\chi^*\!\left(\frac {Nx'}\ep \right)+\ep V^*\!\left(\frac {Nx'}\ep \right)=0
\end{equation}
and the expansion of $\varphi$ near the origin \eqref{eq:varphiparab}, we have 
\begin{equation}
\begin{aligned}
\lefteqn{
-\ep\chi^*\!\left(\frac {Mx}\ep \right)-\ep V^*\!\left(\frac {Mx}\ep \right)
} \quad & \\
& =-\ep\chi^*\!\left(\frac {Nx'+x_dn}\ep \right)-\ep V^*\!\left(\frac {Nx'+x_dn}\ep \right)\\
&=-\ep\chi^*\!\left(\frac {Nx'}\ep \right)-\ep V^*\!\left(\frac {Nx'}\ep \right)-\int_0^1\left[\nabla\chi^*+\nabla V^*\right]\!\left(\frac {Nx'+tx_dn}\ep \right)\cdot nx_d\,dt\\
&=-\int_0^1\left[\nabla\chi^*+\nabla V^*\right]\!\left(\frac{Nx'+t\varphi(x')n}{\ep}\right)\cdot n\,dt\tilde{\varphi}(x')\cdot x'^2.
\end{aligned}
\end{equation}
Denote
\begin{equation*}
F(y',y_d):=\int_0^1\left[\nabla\chi^*+\nabla V^*\right]\!\left(Ny'+ty_dn\right)\cdot n\,dt.
\end{equation*}
By Lemma \ref{lem.estlayer}, all derivatives of $F$ (at least) up to order three are uniformly bounded. Moreover, for any $x=(x',\varphi(x'))$, $|x'|<r_0$,
\begin{equation}\label{eq:parabzero}
\ep\chi^*\!\left(\frac {Mx}\ep \right)+\ep V^*\!\left(\frac {Mx}\ep \right)=F\!\left(\frac{x'}{\ep},\frac{\varphi(x')}{\ep}\right)\tilde{\varphi}(x')\cdot x'^2.
\end{equation}
Therefore, close to the origin, the boundary condition is squeezed between two paraboloids. We will rely on this property, when estimating
\begin{equation*}
s^\ep(x_0)=-\int_{\partial\Omega} P^\ep_\Omega(x_0,x)\left(\ep\chi^*\!\left(\frac {Mx}\ep \right)+\ep V^*\!\left(\frac {Mx}\ep \right)\right)\,d\H^{d-1}(x),
\end{equation*}
for $x_0\in B(0,2r)\cap\Omega$. We split the boundary integral into two parts
\begin{multline}\label{eq:splitint}
\int_{\partial\Omega} P^\ep_\Omega(x_0,x)\left(\ep\chi^*\!\left(\frac {Mx}\ep \right)+\ep V^*\!\left(\frac {Mx}\ep \right)\right)\,d\H^{d-1}(x)\\
=\int_{B(0,4\ep^{\frac12})\cap\partial\Omega} P^\ep_\Omega(x_0,x)\left(\ep\chi^*\!\left(\frac {Mx}\ep \right)+\ep V^*\!\left(\frac {Mx}\ep \right)\right)\,d\H^{d-1}(x)\\+\int_{B(0,4\ep^{\frac12})^c\cap\partial\Omega} P^\ep_\Omega(x_0,x)\left(\ep\chi^*\!\left(\frac {Mx}\ep \right)+\ep V^*\!\left(\frac {Mx}\ep \right)\right)\,d\H^{d-1}(x).
\end{multline}
For the first term on the right-hand side of~\eqref{eq:splitint}, on $B\big(0,4\ep^{\frac12}\big)\cap\partial\Omega$ we use~\eqref{eq:parabzero}, which gives
\begin{equation*}
\begin{aligned}
\lefteqn{
\left|\int_{B(0,4\ep^{\frac12})\cap\partial\Omega} P^\ep_\Omega(x_0,x)\left(\ep\chi^*\!\left(\frac {Mx}\ep \right)+\ep V^*\!\left(\frac {Mx}\ep \right)\right)\,d\H^{d-1}(x)\right|
} \qquad & \\
&\leq C\int_{|x'|\leq 4\ep^{\frac12}}\left|P^\ep_\Omega(x_0,(x',\varphi(x')))\left[\ep\chi^*+\ep V^*\right]\!\left(\frac{Nx'+\varphi(x')n}{\ep}\right)\right|\,dx'\\
&\leq C\int_{|x'|\leq 4\ep^{\frac12}}\frac{{x_0}_d-\varphi({x_0}')}{\big(({x_0}_d-\varphi(x'))^2+|x_0'-x'|^2\big)^{\frac d2}}\left|F\!\left(\frac{x'}{\ep},\frac{\varphi(x')}{\ep}\right)\tilde{\varphi}(x')\cdot x'^2\right|\,dx'\\
&\leq C\int_{|x'|\leq 4\ep^{\frac12}}\frac{({x_0}_d-\varphi({x_0}'))|x'|^2}{\big(({x_0}_d-\varphi(x'))^2+|{x_0}'-x'|^2\big)^{\frac d2}}\,dx'.
\end{aligned}
\end{equation*}
Bounding $|x'|$ by the triangle inequality,
\begin{equation*}
|x'|^2\leq 2|x_0'|^2+2|x_0'-x'|^2,
\end{equation*}
and using the following bounds when appropriate, for all $x\in B(0,2r)\cap\Omega$,
\begin{equation*}
{x_0}_d-\varphi({x_0}')\leq r,\qquad |{x_0}|^2\leq r^2,
\end{equation*}
it now follows from the previous series of inequalities that
\begin{equation*}
\begin{aligned}
\lefteqn{
\left|\int_{B(0,4\ep^{\frac12})\cap\partial\Omega} P^\ep_\Omega({x_0},x)\left(\ep\chi^*\!\left(\frac {Mx}\ep \right)+\ep V^*\!\left(\frac {Mx}\ep \right)\right)\,d\H^{d-1}(x)\right|
} \quad & \\
&\leq Cr\int_{|x'|\leq 4\ep^{\frac12}}\frac{1}{|{x_0}'-x'|^{d-2}}dx'+Cr^2\int_{|x'|\leq 4\ep^{\frac12}}\frac{{x_0}_d-\varphi({x_0}')}{\big(({x_0}_d-\varphi(x'))^2+|{x_0}'-x'|^2\big)^{\frac d2}}\,dx'.
\end{aligned}
\end{equation*}
A direct computation gives
\begin{equation*}
r\int_{|x'|\leq 4\ep^{\frac12}}\frac{1}{|{x_0}'-x'|^{d-2}}\,dx'\leq Cr\ep^{\frac12},
\end{equation*}
and, on the other hand, using that for all ${x_0}\in B(0,2r)\cap\Omega$, for all $x\in B(0,4\ep^{\frac12})\cap\Omega$,
\begin{equation*}
|\varphi({x_0}')-\varphi(x')|\leq \|\nabla\varphi\|_{L^\infty(B(0,4\ep^{\frac 12}))}|{x_0}'-x'|\leq C\ep^{\frac12}|{x_0}'-x'|,
\end{equation*}
we obtain
\begin{equation*}
\begin{aligned}
\lefteqn{
r^2\int_{|x'|\leq 4\ep^{\frac12}}\frac{{x_0}_d-\varphi({x_0}')}{\big(({x_0}_d-\varphi(x'))^2+|{x_0}'-x'|^2\big)^{\frac d2}}\,dx'
} \qquad & \\
&=r^2\int_{|x'|\leq 4\ep^{\frac12}}\frac{{x_0}_d-\varphi({x_0}')}{\big(({x_0}_d-\varphi({x_0}')+\varphi({x_0}')-\varphi(x'))^2+|{x_0}'-x'|^2\big)^{\frac d2}}\,dx'\\
&\leq Cr^2\int_{|x'|\leq 4\ep^{\frac12}}\frac{{x_0}_d-\varphi({x_0}')}{\left(({x_0}_d-\varphi({x_0}'))^2+\big(1-\ep^{\frac 12}\big)^2|{x_0}'-x'|^2\right)^{\frac d2}}\,dx' \\
& \leq Cr^2.
\end{aligned}
\end{equation*}

\smallskip

To estimate the second term on the right-hand side of \eqref{eq:splitint}, we use~\eqref{e.PKB} and that the boundary condition is uniformly bounded by $\ep$:
\begin{equation*}
\begin{aligned}
\lefteqn{
\left|\int_{B(0,4\ep^{\frac12})^c\cap\partial\Omega} P^\ep_\Omega(x_0,x)\left(\ep\chi^*\!\left(\frac {Mx}\ep \right)+\ep V^*\!\left(\frac {Mx}\ep \right)\right)\,d\H^{d-1}(x)\right|
} \quad & \\
&\leq C\ep r\int_{B(0,4\ep^{\frac12})^c\cap\partial\Omega}\frac{1}{(|x|-r)^d}\,d\H^{d-1}(x)\\
&\leq C\ep r\int_{\{B_{d-1}(0,r_0)\times (-r_0,r_0)\}^c\cap\partial\Omega}\frac{1}{|x|^d}\,d\H^{d-1}(x)
+C\ep r\int_{\partial\Omega,\, \ep^{\frac12}<|x|<r_0}\frac{1}{|x|^d}\,d\H^{d-1}(x)\\
&\leq C\big(\ep r+\ep^{\frac12}r\big).
\end{aligned}
\end{equation*}
The above inequalities yield~\eqref{eq:boundstep3}.

\smallskip

\emph{Step 4.} The second term on the right-hand side of \eqref{eq:lipest} is straightforwardly estimated. In order to estimate the $C^{1,\rho}$ semi-norm of the boundary data, we estimate its second-order derivative using \eqref{eq:parabzero}. We have
\begin{equation*}
\nabla\left(F\!\left(\frac{x'}{\ep},\frac{\varphi(x')}{\ep}\right)\right)=\frac{1}{\ep}\nabla'F\!\left(\frac{x'}{\ep},\frac{\varphi(x')}{\ep}\right)+\frac{1}{\ep}\partial_dF\!\left(\frac{x'}{\ep},\frac{\varphi(x')}{\ep}\right)\nabla'\varphi(x'),
\end{equation*}
and
\begin{multline*}
\nabla^2\left(F\!\left(\frac{x'}{\ep},\frac{\varphi(x')}{\ep}\right)\right)=\frac{1}{\ep^2}\nabla'^2F\!\left(\frac{x'}{\ep},\frac{\varphi(x')}{\ep}\right)\\
+\frac{2}{\ep^2}\partial_d\nabla'F\!\left(\frac{x'}{\ep},\frac{\varphi(x')}{\ep}\right)\nabla'\varphi(x')+\frac{1}{\ep^2}\partial_d^2F\!\left(\frac{x'}{\ep},\frac{\varphi(x')}{\ep}\right)\nabla'\varphi(x')\nabla'\varphi(x').
\end{multline*}
Subsequently,
\begin{equation*}
\left|\nabla\left(F\!\left(\frac{x'}{\ep},\frac{\varphi(x')}{\ep}\right)\tilde{\varphi}(x')\cdot x'^2\right)\right|\leq C\left(\frac{r^2}{\ep}+r\right)
\end{equation*}
and
\begin{equation*}
\left|\nabla^2\left(F\!\left(\frac{x'}{\ep},\frac{\varphi(x')}{\ep}\right)\tilde{\varphi}(x')\cdot x'^2\right)\right|
\leq C\left(\frac{r^2}{\ep^2}+\frac{r}{\ep}\right).
\end{equation*}
In the end, we get
\begin{equation}\label{eq:boundstep4}
\left[\ep\chi^*\!\left(\frac{Mx}\ep\right)+\ep V^*\!\left(\frac{Mx}\ep\right)\right]_{C^{1,\rho}(\partial\Omega\cap B(0,2r))}\leq \frac{Cr^2}{\ep^{1+\rho}}.
\end{equation}
Combining~\eqref{eq:boundstep3} and~\eqref{eq:boundstep4} yields the proposition.
\end{proof}

\section{Outline of the rest of the proof}
\label{sec.outline}

In this section, we combine the results of the previous sections to obtain a representation for the error in homogenization at a fixed point $x_0\in\Omega$. 

\smallskip

We first summarize the strategy, which begins with the Poisson formula. We take a triadic cube decomposition of the boundary, so that every cube is ``good" in the sense that there is a nearby boundary point whose normal vector is Diophantine (with a good constant). We then approximate the integral on~$\partial \Omega$ in the Poisson formula by a sum of integrals over the tangent planes to the good boundary points. We make an error in moving from~$\partial \Omega$ to these flat tangent planes which depends on the size of the local cube. We then replace the Poisson kernel for the heterogeneous equation with the two-scale expansion, and use Proposition~\ref{prop:expphieps} to approximate $\omega^\ep_{ij}$ with a smooth, periodic function on $\Rd$. The result is then a sum of integrals on $\R^{d-1}$ of quasiperiodic functions (the product of the boundary condition and the periodic approximation of $\omega^{\ep}_{ij}$ restricted to the plane) multiplied by a smooth function (which is the homogenized Poisson kernel times a cutoff  function on the local cube). This latter integral is close to the average of the periodic function, by the quantitative ergodic theorem. Since the cubes in the partition are small, the sum of the integrals over all the cubes is therefore close to the Poisson formula for the homogenized equation. 

\smallskip

Proceeding with the argument, we fix the length scale $\ep>0$ and take $\delta >0$ to be a small exponent fixed throughout (representing the amount of exponent that we give up in the argument). We first take~$\P_\ep$ to be the collection of triadic cubes given by Proposition~\ref{p.CZdecomposition} applied to the function
\begin{equation*}
F(x):= \ep^{1-\delta} \A^{-1}(x),\quad x \in \partial \Omega,
\end{equation*}
where~$\A(x)$ is the Diophantine constant for the unit vector~$n(x)$, which is normal to the boundary of~$\partial \Omega$ at~$x\in \partial \Omega$. We let $\left\{ \psi_\cu \in C^\infty(\Rd)\,:\, \cu\in\P\right\}$ be the partition of unity given by Corollary~\ref{c.partitionofunity} for the partition $\P_\ep$. For each $x\in\partial\Omega$, we denote by $\cu_{\P_\ep}(x)$ the unique element of $\P_\ep$ containing $x$. By Proposition~\ref{p.CZdecomposition}, for each cube $\cu \in \P_\ep$, there exists $\bar x(\cu) \in 3\cu \cap \partial \Omega$ such that 
 \begin{equation} \label{e.bar x(cu)}
\A(\bar x(\cu)) \geq \frac{\ep^{1-\delta}}{\size(\cu)}\,.
\end{equation}
We also denote
\begin{equation*} \label{}
\Gamma_\ep:= \Omega \cap \left( \bigcup_{\cu\in \P_\ep} 5\cu \right).
\end{equation*}
Proposition~\ref{p.CZdecomposition}(iv) gives us that 
\begin{equation}
\label{e.CZinput}
\# \!\left\{ \cu \in \P_\ep \,:\, \size(\cu) \geq 3^n \right\}
\leq C 3^{-2n(d-1)} \ep^{(1-\delta)(d-1)}. 
\end{equation}
Notice that~\eqref{e.CZinput} implies in particular that the largest cube in $P_\ep$ has size at most $C\ep^{\frac{1-\delta}2}$. Since $F$ is bounded below by $\ep^{1-\delta}$, we have 
\begin{equation}
\label{e.unifboundarylayer}
\forall \cu\in\P_\ep, \quad c\ep^{1-\delta} \leq \size(\cu) \leq C\ep^{\frac{1-\delta}2}. 
\end{equation}

\smallskip

For each $\cu\in\P_\ep$, let us denote by $D(\cu):=D_{n(\bar{x}(\cu))}\left(c(\bar{x}(\cu))\right)$, which is  the half-space tangent to $\Omega$ at $\bar{x}(\cu)$.  Here $c(\bar{x}(\cu))  = \bar{x}(\cu) \cdot n(\bar{x}(\cu))$. The half-space corrector is denoted by $V^*(\cu,\cdot)$, that is, $V^*(\cu,\cdot)$ is the solution of the boundary layer problem
\begin{equation*}
\left\{ 
\begin{aligned}
& -\nabla \cdot\left( \a^*(y)\nabla V^*(\cu,\cdot)\right) = 0   & \mbox{in} & \ D_{n(\bar{x}(\cu))}\left(\frac{c(\bar{x}(\cu))}\ep \right), \\
& V^*(\cu,\cdot) = -\chi^*(y) & \mbox{on} & \ \partial D_{n(\bar{x}(\cu))}\left(\frac{c(\bar{x}(\cu))}\ep \right).
\end{aligned}
\right. 
\end{equation*} 
Moreover, $\omega^\ep(\cu,\cdot)=\omega^\ep(\bar{x}(\cu),\cdot)$ is defined by \eqref{e.deftildeomegaeps}, that is, for $1\leq i,\, j\leq L$,
\begin{multline}
\label{e.oscillatingomega'}
\tilde\omega^\ep_{ij}(\cu,x)=h_{ik}(\bar{x}(\cu)) n(\bar{x}(\cu)) \cdot \nabla \left( \mathbf{p}_{lk}(x)+ \ep\chi^{*}_{lk}\!\left(\frac{x}\ep\right) + \ep V^{*}_{lk}\!\left(\cu,\frac{x}\ep\right) \right) \\
\cdot n(\bar{x}(\cu))\a_{lj}\!\left(\frac {x}\ep \right)n(\bar{x}(\cu))\cdot n(\bar{x}(\cu)).
\end{multline}
In order to make the notation lighter, when it is clear from the context, we will drop the dependence of $\bar{x}(\cu)$ on $\cu$ and write simply $\bar{x}$. 

\smallskip

The goal of this section is to prove the following proposition.

\begin{proposition} 
\label{p.goalest}
There exists a positive constant $C(\delta, d,L,\lambda,\a,g,\Omega)<\infty$ such that, for every $x_0\in\Omega \setminus \Gamma_\ep$,
\begin{multline} \label{e.goalest}
\big| u^\ep(x_0) - \bar{u}(x_0) \big| \\
\leq C\ep^{\frac12}+ C\ep^{-10\delta} \sum_{\cu \in \P_\ep} \left( \frac{\size^3(\cu)}{\ep^2} \wedge 1 \right) \frac{\dist(x_0,\partial \Omega)}{|x_0-\bar{x}(\cu)|^d} \size(\cu)^{d-1}. 
\end{multline}
\end{proposition}

The next section is devoted to estimating the $L^q$ norms of the error term on the right side of \eqref{e.goalest}. For now, we turn to the proof of \eqref{e.goalest}. Let us split the estimate of the left-hand side of \eqref{e.goalest} into several pieces, which will be handled separately: for any $x_0\in\Omega \setminus \Gamma_\ep$,
\begin{multline}\label{e.termsrhs}
\big| u^\ep(x_0) - \bar{u}(x_0) \big|=\left| \int_{\partial \Omega} P^\ep_\Omega(x_0,x) g\left(x,\frac {x}\ep \right) \,d\H^{d-1}(x)\right.\\
\left.-\int_{\partial \Omega} \overline{P}_\Omega(x_0,x) \bar{g}(x) \,d\H^{d-1}(x)\right|
\leq\, \mathrm{T}_1\,+\,\mathrm{T}_2\,+\mathrm{T}_3\,+\mathrm{T}_4\,+\mathrm{T}_5,
\end{multline}
where
\begin{multline*}
\mathrm{T}_1:=\left| \int_{\partial \Omega}P^\ep_\Omega(x_0,x)g\left(x,\frac {x}\ep \right) \,d\H^{d-1}(x)\right.\\
\left.- \sum_{\cu \in\P_\ep} \int_{\partial \Omega} \psi_\cu(x) \overline{P}_\Omega (x_0,x) \tilde\omega^\ep(\cu,x)  g\left(x,\frac {x}\ep \right) \,d\H^{d-1}(x)\right|,
\end{multline*}
\begin{multline*}
\mathrm{T}_2:=\sum_{\cu\in\P_\ep} \bigg|  \int_{\partial \Omega} \psi_\cu(x) \overline{P}_\Omega (x_0,x) \tilde\omega^\ep(\cu,x)  g\left(x,\frac {x}\ep \right) \,d\H^{d-1}(x)\\
-  \int_{\partial D(\cu)} \psi_\cu( \mathrm{proj}(x)) \overline{P}_\Omega (x_0,\mathrm{proj}(x)) \tilde\omega^\ep(\cu, x)  g\left(\bar{x},\frac { x}\ep \right) \,d\H^{d-1}(x)\bigg|,
\end{multline*}
\begin{multline*}
\mathrm{T}_3:=\sum_{\cu\in\P_\ep}\bigg|  \int_{\partial D(\cu)} \psi_\cu(\mathrm{proj}(x)) \overline{P}_\Omega (x_0,\mathrm{proj}(x)) \tilde\omega^\ep(\cu,x)  g\left(\bar{x},\frac {x}\ep \right) \,d\H^{d-1}(x) \\
- \int_{\partial D(\cu)}  \psi_\cu(\mathrm{proj}(x)) \overline{P}_\Omega (x_0,\mathrm{proj}(x))\overline g(\bar{x}) \,d\H^{d-1}(x) \bigg|,
\end{multline*}
\begin{multline*}
\mathrm{T}_4:=\sum_{\cu\in\P_\ep}\bigg| \int_{\partial D(\cu)}  \psi_\cu(\mathrm{proj}(x)) \overline{P}_\Omega (x_0,\mathrm{proj}(x))\overline g(\bar{x}) \,d\H^{d-1}(x) \\ 
- \int_{\partial D(\cu)}  \psi_\cu( \mathrm{proj}(x) ) \overline{P}_\Omega (x_0,\mathrm{proj}(x))\overline g(\mathrm{proj}(x)) \,d\H^{d-1}(x)  \bigg|,
\end{multline*}
and
\begin{multline*}
\mathrm{T}_5:=\bigg|\sum_{\cu\in\P_\ep}\int_{\partial D(\cu)}  \psi_\cu( \mathrm{proj}(x)) \overline{P}_\Omega (x_0,\mathrm{proj}(x))\overline g(\mathrm{proj}(x)) \,d\H^{d-1}(x)\\
-\int_{\partial\Omega}\overline{P}_\Omega (x_0,x)\overline g(x) \,d\H^{d-1}(x)\bigg|.
\end{multline*}
Here we define $\mathrm{proj}(x) = \mathrm{proj}(x,\cu_{\P_\ep}(x))$ to be nearest point of $\partial \Omega$ to $x$ such that $x - \mathrm{proj}(x)$ is a multiple of the vector~$n(\bar{x}(\cu))$. Since $\partial \Omega$ is smooth we have that 
\begin{equation*}
\left| \mathrm{proj}(x) - x \right| \leq C \size^2\left(\cu_{\P_\ep}(x) \right), \quad \left| \nabla \mathrm{proj}(x) - \Idd_d \right| \leq C \size\left(\cu_{\P_\ep}(x) \right),
\end{equation*}
and, for all $k \in \N$,
\begin{equation*} 
\left| \nabla^k \mathrm{proj}(x)\right| \leq C(k,\partial \Omega,d).
\end{equation*}
Thus, for small enough $\ep\in (0,1)$,  $\mathrm{proj}^{-1}(\cdot)$ is a diffeomorphism between $\partial \Omega \cap 5\cu$ and $\mathrm{proj}^{-1}(\partial \Omega \cap 5\cu) \cap \partial D(\cu)$.
The Jacobian determinants satisfy,
\begin{equation*} 
\left| J_{\mathrm{proj}(\cdot)}(x)  - 1 \right| + \left| J_{\mathrm{proj}^{-1}(\cdot)}(x) - 1 \right| \leq C \size(\cu_{\P_\ep}(x)) \ll 1,
\end{equation*}
because $\size(\cu_{\P_\ep}(x)) \leq C \ep^{\frac{1-\delta}{2}}$. 

\smallskip

\emph{Estimate of $\mathrm T_1$.} Using the expansion in~\eqref{eq:formulaKLS} we may write, for $x_0\in \Omega$ and $x\in\partial \Omega$,  the Poisson kernel $P^\ep_\Omega(x_0,x)$ as follows: 
\begin{align}\label{eq:formulaKLS'}
P^\ep_\Omega(x_0,x) \notag
& =\overline{P}_\Omega(x_0,x)\omega^\ep(x)+R^\ep(x_0,x)  \\ \notag
& = \overline{P}_\Omega(x_0,x) \sum_{\cu \in \P_\ep} \psi_{\cu}(x) \omega^\ep(x)+R^\ep(x_0,x) \\ \notag
& = \overline{P}_\Omega(x_0,x) \sum_{\cu \in \P_\ep} \psi_{\cu}(x) \tilde \omega^\ep(\cu,x)  \\   
& \qquad +  \overline{P}_\Omega(x_0,x) \sum_{\cu \in \P_\ep} \psi_{\cu}(x)  \left(\omega^\ep(x) - \tilde\omega^\ep(\cu,x) \right) + R^\ep(x_0,x),
\end{align}
where $R^\ep(x_0,x)$ satisfies the bound~\eqref{e.Repbound}. Observe that 
by Corollary \ref{cor.approxomegaeps}, we have the following bound 
\begin{equation*}
\big| \omega^\ep(x) - \tilde\omega^\ep(\cu,x) \big| \leq C_\rho \left( \ep^{\frac12}+\frac{\size^{2+\rho}\left(\cu_{\P_\ep}\!\left(x\right)\right)}{\ep^{1+\rho}}  \wedge 1 \right),
\end{equation*}
for $\sigma\in(0,1)$. Thus, we can bound the second term on the right of~\eqref{eq:formulaKLS'} for $x_0\in \Omega\setminus \Gamma_\ep$ by
\begin{multline*}
\left| \overline{P}_\Omega(x_0,x)  \psi_{\cu}(x)  \big(\omega^\ep(x) - \tilde\omega^\ep(\cu,x) \big) \right| \\
 \leq \frac{C_\rho\dist(x_0,\partial \Omega)}{ \left| x_0- x \right|^{d}} \left( \ep^{\frac12}+\frac{\size^{2+\rho}\left(\cu_{\P_\ep}\!\left(x\right)\right)}{\ep^{1+\rho}}  \wedge 1 \right). 
\end{multline*}
We therefore obtain from the above estimates and~\eqref{e.Repbound} that 
\begin{multline}
\label{e.combine2}
 \mathrm{T}_1  
 \leq C_\rho \left\| g \right\|_{L^\infty(\partial\Omega\times\Rd)} \sum_{\cu\in\P_\ep} \frac{\dist(x_0,\partial \Omega)}{ \left| x_0- \bar{x}(\cu) \right|^{d}} \size(\cu)^{d-1} \left( \ep^{\frac12}+\frac{\size^{2+\rho}(\cu)}{\ep^{1+\rho}}  \wedge 1\right)  \\
 +  C\ep\left\| g \right\|_{L^\infty(\partial\Omega\times\Rd)}  \int_{\partial \Omega} \left|x_0-x \right|^{-d}\log\left(\ep^{-1}\left|x_0-x \right|+2\right)\,dx.
\end{multline}
Notice that the integral in the second term on the right is bounded by $C$, and therefore the second term is actually bounded by $C$ times the first term and can be neglected. Since 
\begin{equation*}
\sum_{\cu\in\P_\ep} \frac{\dist(x_0,\partial \Omega)}{ \left| x_0- \bar{x}(\cu) \right|^{d}} \size(\cu)^{d-1} \leq C,
\end{equation*}
and $\ep\leq \size(\cu)$ for any cube $\cu\in\P_\ep$, we obtain
\begin{equation}
\label{e.T1}
  \mathrm{T}_1 
\leq C\ep^{\frac12} + C \sum_{\cu\in\P_\ep} \left(\frac{\size^3(\cu)}{\ep^2}\wedge 1\right)  \frac{\dist(x_0,\partial \Omega)}{ \left| x_0- \bar{x}(\cu) \right|^{d}} \size(\cu)^{d-1}.
\end{equation}

\smallskip

\emph{Estimate of $\mathrm T_2$.} We freeze the slow variables and, in each $\cu$, move the integral from $\partial \Omega$ to the boundary of the half-space~$D(\cu)$. Recall first the bound
\begin{equation*} 
\left\| \tilde\omega^\ep(\cu, \cdot)\right\|_{L^\infty} +  \left\|g(\cdot,\cdot)\right\|_{L^\infty} \leq C.
\end{equation*}
Using this, we have, for every $x \in \partial D(\cu) \cap \supp \psi_{\cu}(\mathrm{proj}(\cdot))$, 
\begin{equation*} 
\left| \tilde\omega^\ep(\cu, \mathrm{proj}(x)) -  \tilde\omega^\ep(\cu,x) \right| \leq C \left\|\nabla \tilde \omega^\ep(\cu,\cdot)\right\|_{L^\infty} \size^2(\cu)  \leq C \left( \frac{\size^2(\cu)}{\ep} \wedge 1\right) ,
\end{equation*}
which is due to Lemma~\ref{lem.estlayer}, and 
\begin{multline*} 
\left| g\left(\mathrm{proj}(x),\frac {\mathrm{proj}(x)}\ep \right) - g\left(\bar x(\cu),\frac {x}\ep \right) \right| \\ \leq C  \left( \left\|\nabla_x g \right\|_{L^\infty}\size(\cu) + \left\|\nabla_y g \right\|_{L^\infty}\frac{\size^2(\cu)}{\ep}     \right) \wedge 2 \|Êg\|_{L^\infty} \leq C\left( \frac{\size^2(\cu)}{\ep} \wedge 1\right)  
\end{multline*}
by the regularity of $g$ and by the fact that $\size(\cu) \geq c \ep$. 
Moreover, the Jacobian determinant satisfies 
\begin{equation*} 
\left| J_{\mathrm{proj}^{-1}(\cdot)}(x) - 1 \right| \leq C \size(\cu) \leq  C \left( \frac{\size^2(\cu)}{\ep} \wedge 1 \right).
\end{equation*}
Recording the error we make, for each $x \in \partial D(\cu) \cap \supp \psi_{\cu}(\mathrm{proj}(\cdot))$,
\begin{multline*}
\bigg| \tilde\omega^\ep(\cu,\mathrm{proj}(x))  g\left(\mathrm{proj}(x),\frac {\mathrm{proj}(x)}\ep \right) J_{\mathrm{proj}^{-1}(\cdot)}(x)  
 - \tilde\omega^\ep(\cu,x) g\left(\bar x(\cu),\frac {x}\ep \right) \bigg| \\ \leq C\left( \frac{\size^2(\cu)}{\ep} \wedge 1\right)  .
\end{multline*}
Thus we obtain, treating $\overline{P}_\Omega$ as in the case of the term $T_1$,
\begin{equation}
\label{e.T2}
\mathrm{T}_2 \leq C \sum_{\cu\in\P_\ep} \size(\cu)^{d-1}\frac{\dist(x_0,\partial \Omega)}{ \left| x_0- \bar{x}(\cu) \right|^{d}}\left( \frac{\size^2(\cu)}\ep\wedge1\right).
\end{equation}

\smallskip

\emph{Estimate of $\mathrm T_3$.} Applying the ergodic theorem of Proposition~\ref{p.ergodicthm}, we can compute the integrals over the flat half-spaces $\partial D(\cu)$ for $x_0\in\Omega \setminus \Gamma_\ep$ up to a tiny error. Let us do some computations to fit into the framework of our ergodic theorem. First of all, notice that for all $x\in\partial D(\cu)$, for all $z=M^Tx$ ($M$ is as usual an orthogonal matrix sending $e_d$ to $n(\bar{x}(\cu))$ and $N$ its $d-1$ first columns), for all $1\leq i,\, j\leq L$, we have
\begin{align*}
\lefteqn{
{\tilde{\omega}^\ep}_{ij}(\cu,x)
} \quad & \\
& =  h_{ik}(\bar{x}) n(\bar{x}) \cdot \nabla \bigg( \mathbf{p}_{lk}(x)+ \ep\chi^{*}_{lk}\!\left(\frac{x}\ep\right)  
+ \ep V^{*}_{lk}\!\left(\cu,\frac{x}\ep\right) \bigg)\cdot n(\bar{x})\a_{lj}\!\left(\frac {x}\ep \right)n(\bar{x})\cdot n(\bar{x}) \\
& = h_{ik}(\bar{x})\a_{kj}\!\left(\frac {x}\ep \right)n(\bar{x})\cdot n(\bar{x}) 
+h_{ik}(\bar{x})n(\bar{x})\cdot\nabla\chi^*_{lk}\!\left(\frac {x}\ep \right)\cdot n(\bar{x})\a_{lj}\!\left(\frac {x}\ep \right)n(\bar{x})\cdot n(\bar{x}) \\
& \qquad+h_{ik}(\bar{x})n(\bar{x})\cdot\nabla V^*_{lk}\!\left(\frac {x}\ep \right)\cdot n(\bar{x})\a_{lj}\!\left(\frac {x}\ep \right)n(\bar{x})\cdot n(\bar{x}).
\end{align*}
Now, we have
\begin{equation*}
h_{ik}(\bar{x})\a_{kj}\!\left(\frac {x}\ep \right)n(\bar{x})\cdot n(\bar{x})=h_{ik}(\bar{x})\b^{dd}_{kj}\!\left(\frac {Nz'+c(\bar{x})n(\bar{x})}\ep \right),
\end{equation*}
as well as 
\begin{equation*}
\begin{aligned}
&h_{ik}(\bar{x})n(\bar{x})\cdot\nabla\chi^*_{lk}\!\left(\frac {x}\ep \right)\cdot n(\bar{x})\a_{lj}\!\left(\frac {x}\ep \right)n(\bar{x})\cdot n(\bar{x})\\
&=h_{ik}(\bar{x})n(\bar{x})\cdot \nabla\chi^*_{lk}\!\left(\frac {Nz'+c(\bar{x})n(\bar{x})}\ep \right)\cdot n(\bar{x})\b^{dd}_{kj}\!\left(\frac {Nz'+c(\bar{x})n(\bar{x})}\ep \right),
\end{aligned}
\end{equation*}
and also
\begin{equation*}
\begin{aligned}
&h_{ik}(\bar{x})n(\bar{x})\cdot\nabla V^*_{lk}\!\left(\frac {x}\ep \right)\cdot n(\bar{x})\a_{lj}\!\left(\frac {x}\ep \right)n(\bar{x})\cdot n(\bar{x})\\
&=h_{ik}(\bar{x})e_d\cdot \left(\begin{array}{c}N^T\nabla_\theta\\ \partial_t\end{array}\right) \mathcal V^*_{lk}\!\left(\frac {Nz'}\ep, \frac{c(\bar{x})}{\ep} \right)\cdot n(\bar{x})\b^{dd}_{kj}\!\left(\frac {Nz'+c(\bar{x})n(\bar{x})}\ep \right)\\
&=h_{ik}(\bar{x})\partial_t\mathcal V^*_{lk}\!\left(\frac {Nz'}\ep, \frac{c(\bar{x})}{\ep} \right)\cdot n(\bar{x})\b^{dd}_{kj}\!\left(\frac {Nz'+c(\bar{x})n(\bar{x})}\ep \right),
\end{aligned}
\end{equation*}
where $\mathcal V^*$ is the solution given by Proposition \ref{p.halfspace} to \eqref{eq-bdarylayertorus} with $\b$ replaced by $\b^*$, $a=\frac{c(\bar{x})}{\ep}$ and $\mathcal V_0=-\chi^*\left(\theta+\frac{c(\bar{x})}{\ep} \right)$. Notice that for all $\theta\in\mathbb T^d$, for all $t>a$,
\begin{equation*}
\mathcal V^*(\theta,t)=\widetilde{\mathcal V}^*(\theta+na,t-a),
\end{equation*}
with $\widetilde{\mathcal V}^*$ the unique solution given by Proposition \ref{p.halfspace} to \eqref{eq-bdarylayertorus} with $\b$ replaced by $\b^*$, $a=0$ and $\mathcal V_0=-\chi^*(\theta)$. Let us now compute the integral. We rotate the hyperplane $\partial D(\cu)$ and change the variable. Let 
\begin{equation*}
r:=\size(\cu),\quad \eta:=\frac{\ep}{r}\quad \mbox{and}\quad a:=\frac{c(\bar{x})}\ep. 
\end{equation*}
We have, with $x = r N z' + c(\bar x) n(\bar x)$, $z' \in \R^{d-1}$, 
\begin{equation*}
\begin{aligned}
&\int_{\partial D(\cu)} \psi_\cu(\mathrm{proj}(x)) \overline{P}_\Omega (x_0,\mathrm{proj}(x)) \tilde{\omega}^\ep(\cu,x)  g\left(\bar{x},\frac x\ep \right) \,d\H^{d-1}(x)\\
&=\int_{\R^{d-1}} \Psi_{x_0,r}(z')  h(\bar{x})\left(\Idd_L+n(\bar{x})\cdot \nabla\chi^*\!\left(\frac {Nz'}{\eta}+an(\bar{x}) \right)\cdot n(\bar{x})\right.\\
&\qquad\left.+\partial_t\widetilde{\mathcal V}^*\!\left(\frac {Nz'}{\eta}+an(\bar{x}) ,0\right)\cdot n(\bar{x})\right)\\
&\qquad\qquad\b^{dd}\!\left(\frac {Nz'}{\eta}+an(\bar{x}) \right)g\left(\bar{x},\frac {Nz'}{\eta}+a n(\bar{x})
 \right)\,d z',
\end{aligned}
\end{equation*}
where, for all $z'\in\R^{d-1}$,
\begin{equation*}
\Psi_{x_0,r}(z'):=r^{d-1} \psi_\cu\big(\mathrm{proj}(r N z' + c(\bar{x}) n(\bar x) )\big) \overline{P}_\Omega \big(x_0,\mathrm{proj}(r N z' + c(\bar{x}) n(\bar x)) \big).
\end{equation*}
Applying the ergodic theorem for $\Psi=\Psi_{x_0,r}$, $K$ defined for $\theta\in\R^d$ by
\begin{multline*}
K(\theta):=h(\bar{x})\Big(\Idd_L+e_d\cdot M^T\nabla\chi^*\!\left(\theta+an(\bar{x})\right)\cdot n(\bar{x})\\
+\partial_t\widetilde{\mathcal V}^*\!\left(\theta+an(\bar{x}) ,0\right)\cdot n(\bar{x})\Big)\b^{dd}\!\left(\theta+an(\bar{x}) \right)g\left(\bar{x},\theta+an(\bar{x}) \right),
\end{multline*}
and $f(|\xi|)=|\xi|^{-\kappa}$, we eventually get for all $k\in\N$,
\begin{multline}\label{e.ergothappl}
\left| 
\int_{\R^{d-1}} \Psi_{x_0,r}(z') K\!\left( \frac{Nz'}{\eta} \right)\,dz' 
- \widehat{K}(0) \int_{\R^{d-1}} \Psi_{x_0,r}(z') \, dz' 
\right|  \\
\leq 
\left( A^{-1}(\bar x(\cu)) \eta \right)^k 
\left( \int_{\R^{d-1}} \left| \nabla_{z'}^k \Psi_{x_0,r}(z') \right|\,dz' \right)
\left( \sum_{\xi \in \Zd\setminus\{ 0 \}} 
\left| \widehat{K}(\xi) \right| \left| f(|\xi|) \right|^{-k} \right).
\end{multline}
Notice that using the bound \eqref{e.bar x(cu)}, we have
\begin{equation*}
\A^{-1}(\bar x(\cu))\eta \leq \frac{\size(\cu)}{\ep^{1-\delta}}\frac{\ep}{\size(\cu)}=\ep^\delta.
\end{equation*}
Moreover, 
\begin{multline}\label{e.hatk}
\widehat{K}(0) \\
=\int_{\mathbb T^d}h(\bar{x})\Big(\Idd_L+n(\bar{x})\cdot \nabla\chi^*\!\left(\theta\right)\cdot n(\bar{x})+\partial_t\widetilde{\mathcal V}^*\!\left(\theta,0\right)\cdot n(\bar{x})\Big)\b^{dd}\!\left(\theta \right)g\left(\bar{x},\theta\right)\,d\theta.
\end{multline}
We can now proceed with the definition of $\bar{g}$.

\begin{definition}\label{def.defbarg}
For $x\in\partial\Omega$,
\begin{multline}\label{e.defbarg}
\bar{g}(x)\\
:=\int_{\mathbb T^d}h(x)\Big(\Idd_L+n(x)\cdot \nabla\chi^*\!\left(\theta\right)\cdot n(x)+\partial_t\widetilde{\mathcal V}^*\!\left(\theta,0\right)\cdot n(x)\Big)\b^{dd}\!\left(\theta \right)g\left(x,\theta\right)\,d\theta.
\end{multline}
where $n(x)\in\partial B_1$ is, as usual, the outer unit normal at $x$, and $\widetilde{\mathcal V}^*$ is the associated higher-dimensional boundary layer corrector solving \eqref{eq-bdarylayertorus} with $\b$ replaced by $\b^*$, $a=0$ and $\mathcal V_0=-\chi^*(\theta)$. 
\end{definition}

Notice that for $x=\bar{x}(\cu)$, the definition is consistent with \eqref{e.hatk}, as it should be. Let us point out that we define $\bar{g}$ for any $x\in\partial\Omega$, because the expression on the right-hand side of \eqref{e.defbarg} makes sense whether or not $n(x)$ belongs to $\R\Zd$. That being said, the convergence in the ergodic theorem only holds for Diophantine directions. Moreover, since the complementary of the set of all Diophantine directions with some positive constant $A$ is of measure zero, the only directions that matter are the Diophantine ones.

\smallskip

We finally estimate $\nabla_{z'}^k \Psi_{x_0,r}(z')$. To this end,
\begin{equation*} 
\left|\nabla_x^m \psi_{\cu}(x) \right| \leq  \frac{C(m,d)}{r^m} \quad \mbox{and} \quad 
\left| \nabla_x^m \overline P_{\Omega}(x_0,x) \right| \leq \frac{C(m,\Omega,d,L)}{ |x-x_0|^{d-1+m}},
\end{equation*}
and since $|x-x_0| \geq cr$ we have that 
\begin{equation*} 
\left| \nabla_x^m \left( \psi_{\cu}(\cdot)  \overline P_{\Omega}(x_0,\cdot))(x) \right) \right| \leq  \frac{C(m,\Omega,d,L)}{r^{d-1+m}} 
\end{equation*}
Therefore, by the chain rule and bounds on the derivatives of $\mathrm{proj}(\cdot)$, we get
 \begin{equation*} 
\left| \nabla_{z'}^k \Psi_{x_0,r}(z')\right| \leq C \indc_{\supp \psi_\cu}(r N z' + c(\bar x) n(\bar x)),
\end{equation*}
so that
\begin{equation*}
\int_{\R^{d-1}} \left| \nabla^k \Psi_{x_0,r}(z') \right|\,dz'\leq C.
\end{equation*}
It follows now from \eqref{e.ergothappl} that there exists $C(d,L,\lambda,\Omega,\delta,\kappa)$ such that
\begin{align}
\label{e.appliergo}
\mathrm{T}_3 & \leq
\sum_{\cu\in\P_\ep}\bigg|  \int_{\partial D(\cu)} \psi_\cu(x) \overline{P}_\Omega (x_0,\mathrm{proj}(x)) \tilde\omega^\ep(\cu,x)  g\left(\bar{x}(\cu),\frac {x}\ep \right) \,d\H^{d-1}(x) \\ \notag
& \qquad\qquad - \int_{\partial D(\cu)}  \psi_\cu(x) \overline{P}_\Omega (x_0,\mathrm{proj}(x))\overline g(\bar{x}(\cu)) \,d\H^{d-1}(x) \bigg|  \\ \notag
& \leq C_\delta\ep^{1000}.
\end{align}
Note that the error in~\eqref{e.appliergo} can actually be made arbitrarily small, in the sense that we can have whatever finite power of~$\ep$ we like at the cost of a larger constant~$C$.

\smallskip

\emph{Estimate of $\mathrm T_4$.} We use Proposition~\ref{p.contbarg}, proved below, to see that the homogenized boundary condition $\overline{g}(x)$ is close to $\overline{g}(\bar{x}(\cu))$ for $x\in \partial \Omega \cap 5 \cu$. These differ by an amount depending on $\size(\cu)$. Indeed, by Proposition~\ref{p.contbarg} and~\eqref{e.bar x(cu)},
\begin{align*} \label{}
\left| \overline{g}(x) - \overline{g}(\bar{x}(\cu)) \right|
& \leq C \left( \frac{\size^2(\cu)}{A^{\frac52} (\bar{x}(\cu))} +  \frac{\size(\cu)}{A^{\frac32} (\bar{x}(\cu))} \right)   \leq C \left( \frac{\size^{\frac{9}2}(\cu)}{\ep^{\frac52(1-\delta)}} + \frac{\size^{\frac{5}2}(\cu)}{\ep^{\frac32(1-\delta)}}   \right).
\end{align*}
Using~\eqref{e.unifboundarylayer} to get 
\begin{equation*} 
\frac{\size^{\frac{3}2}(\cu)}{\ep^{\frac12(1-\delta)}} + \frac{\ep^{\frac12(1-\delta)}}{\size^{\frac{1}2}(\cu)} \leq C,
\end{equation*}
and the boundedness of $\overline{g}$, we obtain
\begin{equation} 
\label{e.lipcontrolgbar}
\left| \overline{g}(x) - \overline{g}(\bar{x}(\cu)) \right| 
\leq C  \left( \frac{\size^{3}(\cu)}{\ep^{2}}   \wedge  1 \right).
\end{equation}
Using this, we can now estimate $\mathrm{T}_4$: for every $x_0\in \Omega\setminus \Gamma_\ep$, we have
\begin{equation*}
\mathrm{T}_4 \leq C \sum_{\cu\in\P_\ep}\size(\cu)^{d-1}\frac{\dist(x_0,\partial \Omega)}{ \left| x_0- \bar{x}(\cu) \right|^{d}} \left( \frac{\size(\cu)^3}{\ep^2} \wedge 1 \right).
\end{equation*}

\smallskip

\emph{Estimate of $\mathrm T_5$.} For the last term we change the variables back, and use 
\begin{equation*} 
|J_{\mathrm{proj}(\cdot)}(x) - 1| \leq C \size(\cu) 
\end{equation*}
together with bounds for $\overline P_\Omega$ and the boundedness of $\overline{g}$, to get
\begin{align*}
\mathrm{T}_5  
& = \sum_{\cu\in\P_\ep} \bigg| \int_{\partial D(\cu)}  \psi_\cu(\mathrm{proj}(x)) \overline{P}_\Omega (x_0,\mathrm{proj}(x))\overline g(\mathrm{proj}(x)) \,d\H^{d-1}(x)\\
& \qquad\qquad\qquad\qquad \qquad - \int_{\partial \Omega}  \psi_\cu(x) \overline{P}_\Omega (x_0,x)\overline g(x) \,d\H^{d-1}(x)\bigg| \\
& \leq C \sum_{\cu\in\P_\ep}\size(\cu)^{d}\frac{\dist(x_0,\partial \Omega)}{ \left| x_0- \bar{x}(\cu) \right|^{d}}.
\end{align*}
The term on the right, in turn, is bounded by the errors appearing in estimates for $T_1$ and $T_4$ by~\eqref{e.unifboundarylayer}, completing the proof of Proposition~\ref{p.goalest}. \qed

\smallskip

We complete this section by stating and proving the Sobolev regularity result for the homogenized boundary condition $\bar{g}$ defined in Definition \ref{def.defbarg}. The proposition was used above in the estimate of $T_4$. 

\begin{proposition}\label{p.contbarg}
There exists $0<\nu_0(d)<\infty$ and $0<C(d,L,\lambda,\a,g)<\infty$ such that for any $x_1,\, x_2\in\partial\Omega$ and $n_i := n(x_i)$, $i \in \{1,2\}$, if $n_2$ Diophantine with constant $\A$ and $|n_1-n_2|<\nu_0$, then
\begin{equation}\label{e.contbarg}
|\bar{g}(x_1)-\bar{g}(x_2)|\leq \frac{C|n_1-n_2|}{\A^{\frac32}}\left(1+\frac{|n_1-n_2|}{\A}\right).
\end{equation}
\end{proposition}

\begin{proof}
To show the continuity of the function~$\bar{g}$ defined in \eqref{e.defbarg}, the only thing that remains to be proved is the continuity of $\partial_t\widetilde{\mathcal V}^*$ in~$n$. Take $x_1,\, x_2\in\partial\Omega$, such that 
$n_2:=n(x_2)$ is Diophantine with constant $\A$ and $|n_1-n_2|<\nu_0$. Estimate \eqref{e.estVt} of Proposition \ref{lem.estVt} together with Sobolev's embedding theorem implies that there is a constant $C(d,L,\lambda,\a)<\infty$ such that
\begin{equation*}
\left\|\partial_t\big(\widetilde{\mathcal V_1}^*-\widetilde{\mathcal V_2}^*\big)\right\|_{L^\infty(\mathbb T^d\times\{0\})}\leq C\frac{|n_1-n_2|}{\A^{\frac32}}\left(1+\frac{|n_1-n_2|}{\A}\right).
\end{equation*}
This concludes the proof of Proposition \ref{p.contbarg}.
\end{proof}

Notice that in the statement of the previous proposition, nothing is assumed of the direction $n_1$, which may be arbitrary in $\partial B_1$. As a result of Proposition~\ref{p.contbarg}, we obtain some regularity on the function $\bar{g}$.

\begin{proposition}
\label{p.gbarregularity}
Suppose that $d>2$.  Then the function~$\bar{g}$ satisfies 
\begin{equation*} 
\nabla \bar{g} \in L^{\frac{2(d-1)}3,\infty}(\partial \Omega).
\end{equation*}
If $d=2$, then $\bar g \in W^{s,1}(\partial \Omega)$ for all $s\in \left(0,\frac23 \right)$.
\end{proposition}

\begin{proof}
\emph{Step 1.} 
We fix $\tau>0$ and apply Proposition~\ref{p.CZdecomposition} to the function $F(y):= \tau A^{-1}(y)$. Note that since $A\leq 1$, we have that $F \geq \tau>0$. The conclusion of the proposition gives us a collection~$\P_\tau$ of triadic cubes satisfying, for every $\cu\in\P_\tau$, 
\begin{equation}
\label{e.Pdeltasnap}
\inf_{x \in 3\cu \cap \partial \Omega} A^{-1}(x) \leq \tau^{-1} \size(\cu)
\end{equation}
as well as
\begin{equation}
\label{e.Pdeltacount}
\# \!\left\{ \cu \in \P_\tau\,:\, \size(\cu) \geq 3^n \right\}
\leq C 3^{-n(d-1)} \H^{d-1}\left( \left\{ x \in \partial \Omega \,:\, \tau A^{-1}(x)\geq 3^{n-2} \right\} \right)
\end{equation}
and, for every $\cu,\cu'\in \P$ such that $\dist(\cu,\cu') = 0$,  
\begin{equation*} \label{}
\frac13 \leq \frac{\size(\cu)}{\size(\cu')} \leq 3. 
\end{equation*}
Observe that~\eqref{e.Pdeltacount} and \eqref{e.Diophest} imply that 
\begin{equation*}
c\tau \leq \sup_{\cu\in \P_\tau} \size(\cu) \leq C\tau^{\frac12}
\end{equation*}
and thus, by~\eqref{e.Pdeltasnap}, for every $\cu\in\P_\tau$,
\begin{equation*}
\inf_{x \in 3\cu \cap \partial \Omega} A^{-1}(x) \leq C\tau^{-\frac12} \leq C\size^{-1}(\cu). 
\end{equation*}
That is, for every $\cu\in\P_\tau$, we have 
\begin{equation}
\label{e.sizesmash}
\size(\cu) \leq C \left( \inf_{x \in 3\cu \cap \partial \Omega} A^{-1}(x) \right)^{-1}.
\end{equation}
We take $\left\{ \psi_{\cu} \right\}_{\cu\in\P_\tau}$ to be the partition of unity given in Corollary~\ref{c.partitionofunity}. For each $\cu\in\P_\tau$, let $\bar{x}(\cu)\in 3\cu \cap \partial \Omega$ be such that 
\begin{equation*}
A^{-1}(\bar{x}(\cu)) \leq 2 \inf_{x \in 3\cu \cap \partial \Omega} A^{-1}(x).
\end{equation*}
Define
\begin{equation*}
\bar{g}_\tau(x):= \sum_{\cu\in\P_\tau} \bar{g}\left( \bar{x} (\cu) \right) \psi_\cu(x). 
\end{equation*}
According to Proposition~\ref{p.contbarg},~\eqref{e.Pdeltasnap} and the triangle inequality, we have, for every $\cu\in\P_\tau$,
\begin{multline*}
\sup_{x\in 5\cu\cap \partial \Omega} \left| \bar{g}(x) - \bar{g}_\tau(x) \right| \\
\leq C\size(\cu)\left( \inf_{x \in 3\cu \cap \partial \Omega} A^{-1}(x) \right)^{\frac32} + C\size^2(\cu) \left( \inf_{x \in 3\cu \cap \partial \Omega} A^{-1}(x) \right)^{\frac52}.
\end{multline*}
Applying~\eqref{e.sizesmash}, we find that the second term on the right side is bounded by $C$ times the one on the left, so we get
\begin{align}
\label{e.ggdeltaconv}
\sup_{x\in 5\cu\cap \partial \Omega} \left| \bar{g}(x) - \bar{g}_\tau(x) \right| 
& \leq C\size(\cu) \left( \inf_{x \in 3\cu \cap \partial \Omega} A^{-1}(x) \right)^{\frac32} \\ \notag
& \leq C\tau^{\frac12} \left( \inf_{x \in 3\cu \cap \partial \Omega} A^{-1}(x) \right)^{\frac32}.
\end{align}
Moreover, using $\left| \nabla \psi_\cu \right| \leq C\size^{-1}(\cu)$, we find that  
\begin{equation}
\label{e.gdeltabnded}
\sup_{x\in\cu} \left| \nabla \bar{g}_\tau(x) \right| \leq C \left( \inf_{x \in 3\cu \cap \partial \Omega} A^{-1}(x) \right)^{\frac32}. 
\end{equation}
The estimate~\eqref{e.ggdeltaconv} implies, for $d>2$, that, as $\tau \to 0$,
\begin{equation*}
\bar{g}_\tau \rightarrow \bar{g} \quad \mbox{in} \ L^p(\partial \Omega), \quad \mbox{for every} \ p \in \left[1,\frac{2(d-1)}3 \right). 
\end{equation*}
Meanwhile, the estimate~\eqref{e.gdeltabnded} implies that the sequence $\{\nabla\bar{g}_\tau \}_{\tau>0}$ is pointwise dominated by a function belonging to $L^{\frac{2(d-1)}3,\infty}(\partial \Omega)$. Thus, in particular, 
\begin{equation*}
\sup_{\tau>0} \left\| \nabla\bar{g}_\tau \right\|_{L^{\frac{2(d-1)}3,\infty}(\partial \Omega)} \leq C <\infty.
\end{equation*}
Thus $\bar{g} \in W^{1,p}(\partial \Omega)$ for every $p< \frac{2(d-1)}3$ and $d>2$, and $\nabla \bar{g} \in L^{\frac{2(d-1)}3,\infty}(\partial \Omega)$. 

\smallskip

\emph{Step 2.} Now we prove the estimate in the case $d=2$. Fix $s\in \left(0,\frac23 \right)$.
Then, using the results of the first step, together with the property 
that if $\dist(\cu,\cu') = 0$, then 
\begin{equation*} 
\frac13 \leq \frac{\size(\cu)}{\size(\cu')} \leq 3\,,
\end{equation*}
we obtain
\begin{equation*} 
\left|\bar g_\tau(x)-\bar g_\tau(y)\right| \leq C \left( |x-y| \left( \inf_{z \in \cu_{\P_\ep} (x) \cap \partial \Omega} A^{-1}(z) \right)^{\frac{3}{2}} \wedge 1 \right)
\end{equation*}
Notice that 
\begin{equation*} 
|x-y| \left( \inf_{z \in \cu_{\P_\ep} (x) \cap \partial \Omega} A^{-1}(z) \right)^{\frac{3}{2}}  \geq 1 \quad \implies |x-y|^{-\theta} \leq  
\left( \inf_{z \in \cu_{\P_\ep} (x) \cap \partial \Omega} A^{-1}(z) \right)^{\frac{3\theta }{2}}
\end{equation*}
and thus we may estimate, for any $\theta > s$, 
\begin{align*} 
\lefteqn{ \int_{\partial \Omega} \int_{\partial \Omega} \frac{|\bar g_\tau(x)-\bar g_\tau(y)|}{|x-y|^{d+s}} \, dx \,dy } \qquad &
 \\ & \leq C \sum_{\cu}  \int_{\partial \Omega} \int_{\partial \Omega \cap \cu} 
 \left( |x-y| \left( \inf_{z \in \cu \cap \partial \Omega} A^{-1}(z) \right)^{\frac{3}{2}} \wedge 1 \right)^\theta \, \frac{dx \,dy}{|x-y|^{d+s}}  \\ 
  & \leq C \sum_{\cu}  \int_{\partial \Omega} \int_{\partial \Omega \cap \cu} 
\left( \inf_{z \in \cu \cap \partial \Omega} A^{-1}(z) \right)^{\frac{3\theta }{2}}\, \frac{dx \,dy}{|x-y|^{d+s-\theta}} \\ 
  & \leq C  \int_{\partial \Omega} A^{-\frac{3\theta}{2}}(x) \, dx,
\end{align*}
where the last estimate follows by the fact that $\theta>s$. The result thus follows by the pointwise convergence of $\bar g_\tau$ to $\bar g$ and Fatou's lemma, provided that $\theta \in \left(s,\frac 23 \right)$, because $A^{-1} \in L^{1,\infty}$ when $d=2$.
\end{proof}

\begin{remark}[On the improvement by Shen and Zhuge]
Let us comment on the upgrade of the regularity of $\nabla\bar{g}$ by Shen and Zhuge in \cite{ShenZhuge_arXiv16}. Using a weighted estimate, they are able to refine the bounds on $\mathcal V=\mathcal V_1-\mathcal V_2$ in a layer close the boundary in the following way (see \cite[equation (6.11)]{ShenZhuge_arXiv16}): for all $\sigma\in(0,1)$, there exists a constant $C(d,L,\lambda,\a,\sigma)<\infty$ such that
\begin{equation*}
\int_{\mathbb T^d}\left|N^T\nabla_\theta\left(\mathcal V_1(\theta,0)-\mathcal V_2(\theta,0)\right)\right|d\theta\leq \frac{C|n_1-n_2|}{\A^{1+\sigma}}\left(1+\frac{|n_1-n_2|}{\A}\right).
\end{equation*}
Consequently, they can prove (see \cite[Theorem 6.1]{ShenZhuge_arXiv16}) that for any $\sigma\in(0,1)$, there exists a constant $C(d,L,\lambda,\a,g,\sigma)<\infty$ such that for any $x_1,\, x_2\in\partial\Omega$ and $n_i := n(x_i)$, $i \in \{1,2\}$, if $n_2$ Diophantine with constant $\A$ and $|n_1-n_2|<\nu_0$, then
\begin{equation*}
|\bar{g}(x_1)-\bar{g}(x_2)|\leq \frac{C|n_1-n_2|}{\A^{1+\sigma}}\left(1+\frac{|n_1-n_2|}{\A}\right).
\end{equation*}
Following our proof of Proposition \ref{p.gbarregularity}, this yields the improved regularity $\nabla\bar{g}\in L^q$ for any $q<d-1$ in dimension $d\geq 3$, and $\bar{g}\in W^{s,1}$ for any $s<1$ in dimension $d=2$. Moreover, the homogenization error \eqref{e.goalest} is improved to
\begin{multline} \label{e.goalestimproved}
\big| u^\ep(x_0) - \bar{u}(x_0) \big| \\
\leq C\ep^{\frac12}+ C_\sigma\ep^{-10\delta} \sum_{\cu \in \P_\ep} \left( \frac{\size^{2+\sigma}(\cu)}{\ep^{1+\sigma}} \wedge 1 \right) \frac{\dist(x_0,\partial \Omega)}{|x_0-\bar{x}(\cu)|^d} \size(\cu)^{d-1},
\end{multline}
for any $\sigma>0$, with a constant $C_\sigma$ depending on $\sigma$.
\end{remark}

\section{Estimate of the boundary integral}
\label{s.estimates}

In the previous section, we encountered the following function, which represents the error in homogenization at a point $x_0\in \Omega \setminus \Gamma_\ep$:
\begin{equation}
\label{e.error}
E_\ep(x_0):= \sum_{\cu \in \P_\ep} \left( \frac{\size^3(\cu)}{\ep^2} \wedge 1 \right) \frac{\dist(x_0,\partial \Omega)}{|x_0-\bar{x}(\cu)|^d} \size(\cu)^{d-1},
\end{equation}
where $\Gamma_\ep$ denotes the boundary layer 
\begin{equation*}
\Gamma_\ep : = \Omega \cap \left( \bigcup_{\cu\in \P_\ep} 5\cu \right).
\end{equation*}
Here $\P_\ep$ is the collection of triadic cubes given by Proposition~\ref{p.CZdecomposition} for the function
\begin{equation*}
F(x):= \ep^{1-\delta} \A^{-1}(x),\quad x \in \partial \Omega,
\end{equation*}
$\delta\in\left(0,\frac1{50}\right)$ is a tiny, fixed exponent and~$\A(x)$ is the Diophantine constant for the unit vector~$n(x)$ which is normal to the boundary of~$\partial \Omega$ at~$x\in \partial \Omega$. Note that $\A$ is bounded from above by~1 and thus $F$ is bounded from below by a positive constant (namely $\ep^{1-\delta}$) and therefore Proposition~\ref{p.CZdecomposition} applies.

\smallskip

Observe that, by~\eqref{e.Diophest}, for every $n\in\N$ we have 
\begin{equation*}
\H^{d-1} \left( \left\{ x \in \partial \Omega \,:\, \ep^{1-\delta}\A^{-1} > 3^{n-2} \right\} \right) \leq C\left( \ep^{-(1-\delta)} 3^{n} \right)^{1-d} = C\ep^{(1-\delta)(d-1)} 3^{-n(d-1)}.
\end{equation*}
Therefore, applying Proposition~\ref{p.CZdecomposition}(iv) gives us that 
\begin{equation}
\label{e.CZinput2}
\# \!\left\{ \cu \in \P_\ep \,:\, \size(\cu) \geq 3^n \right\}
\leq C 3^{-2n(d-1)} \ep^{(1-\delta)(d-1)}. 
\end{equation}
It is easy to see this is equivalent to the statement that, for every $t>0$,
\begin{equation}
\label{e.sizewd-1}
\H^{d-1} \left(  \left\{ x\in \partial \Omega\,:\, \size(\cu_{\P_\ep}(x)) \geq t \right\} \right) \leq C \ep^{(1-\delta)(d-1)} t^{1-d}.
\end{equation}
In other words, $x\mapsto \size\left( \cu_{\P_\ep}(x)\right)$ belongs to $L^{d-1,\infty}(\partial\Omega)$ with the norm
\begin{equation} \label{e.sizewd-2}
\left\| \size(\cu_{\P_\ep}(\cdot)) \right\|_{L^{d-1,\infty}(\partial \Omega)} \leq C \ep^{1-\delta}.
\end{equation}
Notice that~\eqref{e.CZinput2} implies in particular that the largest cube in $P_\ep$ has size at most $C\ep^{\frac{1-\delta}2}$. Since $F$ is bounded below by $\ep^{1-\delta}$, we have 
\begin{equation}
\label{e.unifboundarylayer2}
\forall \cu\in\P_\ep, \quad c\ep^{1-\delta} \leq \size(\cu) \leq C\ep^{\frac{1-\delta}2}. 
\end{equation}
From these estimates and interpolation we see that (even in $d=2$)
\begin{equation}
\left\| \size(\cu_{\P_\ep}(\cdot)) \right\|_{L^{1}(\partial \Omega)} \leq C \ep^{1-2\delta}
\end{equation}
which implies
\begin{equation}
\label{e.gammaep}
\left| \Gamma_\ep \right|
 \leq C\ep^{1-2\delta}.
\end{equation}
The main purpose of this section is to estimate the $L^q$ norms of $E_\ep$ outside of the boundary layer $\Gamma_\ep$. 

\begin{lemma}
\label{l.layeryourcake}
We have
\begin{equation}
\label{e.Eepbounds}
\left\| E_\ep \right\|_{L^q(\Omega\setminus \Gamma_\ep)}^q \leq   C \cdot \left\{ 
 \begin{aligned}
  &  \ep^{\frac13 -2\delta} & \mbox{in} & \ d= 2,  \ q\in \left[ 1,\infty\right],  \\
 &  \ep^{ \frac23 -3\delta } & \mbox{in} & \ d= 3,  \ q\in \left[ 1,\infty\right],  \\
 &  \ep^{1-7\delta} & \mbox{in} & \ d\geq 4, \ q\in \left[ 1,\frac{d-1}3 \right], \\
 &  \ep^{\frac{d-1}{3}(1-7\delta)} & \mbox{in} & \ d\geq 4, \ q\in \left[ \frac{d-1}3,\infty \right].
  \end{aligned} 
 \right.
\end{equation}
\end{lemma}
\begin{proof}
We begin by rewriting $E_\ep(x_0)$ in the following way. We take $\Gamma_\ep$ to be the boundary layer given by
\begin{equation*}
\Gamma_\ep:= \Omega \cap \left( \bigcup_{\cu\in \P_\ep} 2\cu \right). 
\end{equation*}
Then for each $x_0\in \Omega \setminus \Gamma_\ep$ and $\cu\in \P_\ep$, we have
\begin{equation*}
\max_{x\in \partial \Omega \cap \cu} |x_0-x| \leq C \min_{x\in   \partial \Omega \cap \cu} |x_0-x|.
\end{equation*}
Thus we have, for each $x_0\in \partial \Omega_r \setminus \Gamma_\ep$,
\begin{align}
\label{e.layercakejack}
E_\ep(x_0) 
& = \sum_{\cu \in \P_\ep} \left( \frac{\size^3(\cu)}{\ep^2}\wedge 1 \right) \frac{\dist(x_0,\partial \Omega)}{|x_0-\bar{x}(\cu)|^d} \size(\cu)^{d-1}
\\
& \leq Cr\ep^{-2} \int_{\partial \Omega} \left( \size^3(\cu_{\P_\ep}(x))\wedge \ep^2 \right) \frac{1}{|x_0-x|^d}\,d\H^{d-1}(x) \notag \\
& \leq C  \sum_{m=0}^{\lceil C|\log r| \rceil} (2^mr)^{1-d}\int_{\partial \Omega \cap B_{2^{m+1}r}(x_0)} \left( \frac{\size^3(\cu_{\P_\ep}(x))}{\ep^2}\wedge 1 \right) \,d\H^{d-1}(x). \notag
\end{align}
Denote the $m$-th summand by
\begin{equation*}
E_{\ep,r,m}(x_0):= (2^mr)^{1-d}\ep^{-2}\int_{\partial \Omega \cap B_{2^{m+1}r}(x_0)} \left( \size^3(\cu_{\P_\ep}(x))\wedge \ep^2 \right) \,d\H^{d-1}(x).
\end{equation*}
We proceed by estimating, for each fixed $m\in\N$, the $L^q$ norm of each of the functions $E_{\ep,r,m}$ for each $q\in [1,\infty)$. We compute 
\begin{align}
\label{e.Eermsetup}
\lefteqn{
\int_{\partial \Omega_r} \left| E_{\ep,r,m}(x_0) \right|^q \,d\H^{d-1}(x_0)
} \qquad & \\ \notag
& \leq C(2^mr)^{1-d} \int_{\partial \Omega_r} \int_{\partial\Omega \cap B_{2^{m+1}}(x_0)}  \left( \frac{\size^3(\cu_{\P_\ep}(x))}{\ep^2} \wedge 1 \right)^q\,d\H^{d-1}(x)\,d\H^{d-1}(x_0) \\ \notag
& \leq C \int_{\partial \Omega} \left( \frac{\size^3(\cu_{\P_\ep}(x))}{\ep^2} \wedge 1 \right)^q\,d\H^{d-1}(x).
\end{align}
The estimate of the integral on the right side is now split into cases depending on the exponent~$q$ and the dimension~$d$. For convenience, denote
\begin{equation*}
G_\ep(x):= \size^3(\cu_{\P_\ep}(x)) \wedge {\ep^2}, \quad x\in\partial\Omega.
\end{equation*}
The claim is that  
\begin{equation}
\label{e.Gepbounds}
\ep^{-2}\left\| G_\ep \right\|_{L^q(\partial \Omega)}^q 
 \leq  \left\{ 
 \begin{aligned}
  & C \ep^{\frac13 -\delta} & \mbox{in} & \ d= 2,  \ q\in \left[ 1,\infty\right],  \\
 & C \ep^{\frac23 -2\delta} & \mbox{in} & \ d= 3,  \ q\in \left[ 1,\infty\right],  \\
 & C \ep^{1-6\delta} & \mbox{in} & \ d\geq 4, \ q\in \left[ 1,\frac{d-1}3 \right], \\
 & C \ep^{\frac{d-1}{3}(1-6\delta)} & \mbox{in} & \ d\geq 4, \ q\in \left[ \frac{d-1}3,\infty \right].
  \end{aligned} 
 \right.
\end{equation}
Using~\eqref{e.CZinput2} and~\eqref{e.unifboundarylayer2}, we compute to obtain, for $d\geq 4$,
\begin{align*}
\left\| G_\ep \right\|_{L^{\frac{d-1}3}(\partial \Omega)}^{\frac{d-1}3}
& = \int_{0}^{\ep^2} t^{\frac{d-1}3-1}\H^{d-1} \left( \left\{ x\in \partial \Omega\,:\, G_\ep(x) \geq t \right\} \right) \,dt \\
& = \int_{c\ep^3}^{\ep^2} t^{\frac{d-1}3-1} \H^{d-1} \left( \left\{ x\in \partial \Omega \,:\, \size(\cu_{\P_\ep} (x))^3 > t \right\} \right) \,dt \\
& \leq C\ep^{(1-\delta)(d-1)} \int_{c\ep^3}^{\ep^2} t^{-1} \,dt \\
& = C \ep^{(1-\delta)(d-1)} \left| \log\ep \right| \\
& \leq C \ep^{(1-2\delta)(d-1)}.
\end{align*}
Thus 
\begin{equation*}
\left\| G_\ep \right\|_{L^{\frac{d-1}3}(\partial \Omega)} \leq C \ep^{3(1-2\delta)}.
\end{equation*}
This yields the third line of~\eqref{e.Gepbounds}. 
Since we trivially have the bound 
\begin{equation}
\label{e.Gepinfty}
\left\| G_\ep \right\|_{L^\infty(\partial\Omega)} \leq \ep^2,
\end{equation}
interpolation gives the last line of~\eqref{e.Gepbounds}. 
In dimension $d=3$, we have 
\begin{align*}
\left\| G_\ep \right\|_{L^{1}(\partial \Omega)}
& = \int_{0}^{\ep^2} \H^{d-1} \left( \left\{ x\in \partial \Omega \,:\,  \size(\cu_{\P_\ep} (x))^3 > t \right\} \right) \,dt \\
& \leq C\ep^{2(1-\delta)} \int_{0}^{\ep^2} t^{-\frac23} \,dt \\
& = C \ep^{\frac83-2\delta},
\end{align*}
while, in dimension $d=2$, a similar computation gives
\begin{align*}
\left\| G_\ep \right\|_{L^{1}(\partial \Omega)}
& = \int_{0}^{\ep^2} \H^{d-1} \left( \left\{ x\in \partial \Omega \,:\,  \size(\cu_{\P_\ep} (x))^3 > t \right\} \right) \,dt \\
& \leq C\ep^{1-\delta} \int_{0}^{\ep^2} t^{-\frac13} \,dt \\
& = C \ep^{\frac 73-\delta}.
\end{align*}
The previous two displays, interpolation and~\eqref{e.Gepinfty} give us the first two lines of~\eqref{e.Gepbounds} and completes the demonstration of~\eqref{e.Gepbounds}.

\smallskip

Combining~\eqref{e.layercakejack},~\eqref{e.Eermsetup} and~\eqref{e.Gepbounds}, we obtain
\begin{equation}
\label{e.Eepboundsp}
\left\| E_\ep \right\|_{L^q(\partial\Omega_r\setminus \Gamma_\ep)}^q \leq C \left| \log r\right|^q  \cdot  \left\{ 
 \begin{aligned}
  &  \ep^{\frac13 -\delta} & \mbox{in} & \ d= 2,  \ q\in \left[ 1,\infty\right],  \\
 &  \ep^{ \frac23 -2\delta } & \mbox{in} & \ d= 3,  \ q\in \left[ 1,\infty\right],  \\
 &  \ep^{1-6\delta} & \mbox{in} & \ d\geq 4, \ q\in \left[ 1,\frac{d-1}3 \right], \\
 &  \ep^{\frac{d-1}{3}(1-6\delta)} & \mbox{in} & \ d\geq 4, \ q\in \left[ \frac{d-1}3,\infty \right].
  \end{aligned} 
 \right.
\end{equation}
Since $\ep \leq r \leq \diam(\Omega) \leq C$, we replace $\left|\log r\right|$ by~$\left|\log \ep\right|$ and then discard the logarithm by giving up some of the exponent to obtain
\begin{equation}
\label{e.Eepboundsp2}
\left\| E_\ep \right\|_{L^q(\partial\Omega_r\setminus \Gamma_\ep)}^q \leq C \cdot \left\{ 
 \begin{aligned}
  &  \ep^{\frac13 -2\delta} & \mbox{in} & \ d= 2,  \ q\in \left[ 1,\infty\right],  \\
 &  \ep^{ \frac23 -3\delta } & \mbox{in} & \ d= 3,  \ q\in \left[ 1,\infty\right],  \\
 &  \ep^{1-7\delta} & \mbox{in} & \ d\geq 4, \ q\in \left[ 1,\frac{d-1}3 \right], \\
 &  \ep^{\frac{d-1}{3}(1-7\delta)} & \mbox{in} & \ d\geq 4, \ q\in \left[ \frac{d-1}3,\infty \right].
  \end{aligned} 
 \right.
\end{equation}
Integrating over all $\ep \leq r \leq \diam(\Omega)$ gives us~\eqref{e.Eepbounds}.
\end{proof}

We now give the proof of the main result. 

\begin{proof}[{Proof of Theorem~\ref{t.main}}]
Fix $q\in [2,\infty]$. Using~\eqref{e.gammaep} and the Agmon-type $L^\infty$ bounds for~$u^\ep$ given in~\cite[Theorem 3(ii)]{AL1}), we have
\begin{equation*}
\left\| u^\ep - u \right\|_{L^q( \Gamma_\ep)}^q \leq \left| \Gamma_\ep \right| \left( \left\| u^\ep \right\|_{L^\infty(\Omega)} +  \left\| u \right\|_{L^\infty(\Omega)}\right)^q \leq C \ep^{1-2\delta}.
\end{equation*}
According to Proposition~\ref{p.goalest} and Lemma~\ref{l.layeryourcake}, we have that, for every $q\in [2,\infty)$, 
\begin{align*}
\left\| u^\ep - u \right\|_{L^q(\Omega\setminus \Gamma_\ep)}^q 
 \leq C\ep^{\frac q2}+ C \left\| E_\ep \right\|^q_{L^q(\Omega\setminus \Gamma_\ep)} 
\leq C\ep^{-10\delta} \cdot \left\{ 
 \begin{aligned}
  &  \ep^{\frac13} & \mbox{in} & \ d= 2, \\
 &  \ep^{ \frac23} & \mbox{in} & \ d= 3, \\
 &  \ep^{1} & \mbox{in} & \ d\geq 4. \\
  \end{aligned} 
 \right.
\end{align*}
Adding the previous two displays and shrinking $\delta$ gives the theorem.
\end{proof}

\smallskip

\begin{remark}[On the optimal exponents in $d=2$ and $3$]
With the remark of Zhongwei Shen, the error term becomes 
\begin{equation}
\widetilde{E}_\ep(x_0):= C_\sigma\sum_{\cu \in \P_\ep} \left( \frac{\size^{2+\sigma}(\cu)}{\ep^{1+\sigma}} \wedge 1 \right) \frac{\dist(x_0,\partial \Omega)}{|x_0-\bar{x}(\cu)|^d} \size(\cu)^{d-1},
\end{equation}
with $C_\sigma$ a constant depending on $\sigma\in(0,1)$. It is now clear from the above computations, modified to deal with the error term $\widetilde{E}_\ep(x_0)$ instead of $E_\ep(x_0)$, that we can get the rates stated in \eqref{e.optrateslowd} in dimensions $d=2$ and $3$.
\end{remark}

\smallskip

\noindent \textbf{Funding and conflict of interest.} The second author was supported by the Academy of Finland project \#258000. The fourth author was partially funded by a PEPS ``Jeune Chercheur'' project of the French CNRS. The authors declare that they have no conflict of interest. 

\smallskip

\noindent \textbf{Acknowledgement.} The authors would like to thank Zhongwei Shen for kindly pointing out to us, after the first version of this paper appeared, that our method also gives the optimal convergence rates in dimensions $d=2$ and $3$.

\small
\bibliographystyle{abbrv}
\bibliography{boundarylayers}

\newcommand{\noop}[1]{}
\begin{thebibliography}{10}

\bibitem{ASS1}
H.~Aleksanyan, H.~Shahgholian, and P.~Sj{\"o}lin.
\newblock Applications of {F}ourier analysis in homogenization of {D}irichlet
  problem {I}. {P}ointwise estimates.
\newblock {\em J. Differential Equations}, 254(6):2626--2637, 2013.

\bibitem{ASS2}
H.~Aleksanyan, H.~Shahgholian, and P.~Sj{\"o}lin.
\newblock Applications of {F}ourier analysis in homogenization of the
  {D}irichlet problem: {$L\sp p$} estimates.
\newblock {\em Arch. Ration. Mech. Anal.}, 215(1):65--87, 2015.

\bibitem{AA}
G.~Allaire and M.~Amar.
\newblock Boundary layer tails in periodic homogenization.
\newblock {\em ESAIM Control Optim. Calc. Var.}, 4:209--243 (electronic), 1999.

\bibitem{AGK}
S.~Armstrong, A.~Gloria, and T.~Kuusi.
\newblock Bounded correctors in almost periodic homogenization.
\newblock {\em Arch. Ration. Mech. Anal.}, 222:393--426, 2016.

\bibitem{AL1}
M.~Avellaneda and F.-H. Lin.
\newblock Compactness methods in the theory of homogenization.
\newblock {\em Comm. Pure Appl. Math.}, 40(6):803--847, 1987.

\bibitem{AL2}
M.~Avellaneda and F.-H. Lin.
\newblock {$L^p$} bounds on singular integrals in homogenization.
\newblock {\em Comm. Pure Appl. Math.}, 44(8-9):897--910, 1991.

\bibitem{BLP}
A.~Bensoussan, J.-L. Lions, and G.~Papanicolaou.
\newblock {\em Asymptotic analysis for periodic structures}.
\newblock AMS Chelsea Publishing, Providence, RI, 2011.
\newblock Corrected reprint of the 1978 original [MR0503330].

\bibitem{DM}
G.~Dolzmann and S.~M{\"u}ller.
\newblock Estimates for {G}reen's matrices of elliptic systems by {$L\sp p$}
  theory.
\newblock {\em Manuscripta Math.}, 88(2):261--273, 1995.

\bibitem{F}
W.~M. Feldman.
\newblock Homogenization of the oscillating {D}irichlet boundary condition in
  general domains.
\newblock {\em J. Math. Pures Appl. (9)}, 101(5):599--622, 2014.

\bibitem{GVM1}
D.~G{\'e}rard-Varet and N.~Masmoudi.
\newblock Homogenization in polygonal domains.
\newblock {\em J. Eur. Math. Soc. (JEMS)}, 13(5):1477--1503, 2011.

\bibitem{GVM2}
D.~G{\'e}rard-Varet and N.~Masmoudi.
\newblock Homogenization and boundary layers.
\newblock {\em Acta Math.}, 209(1):133--178, 2012.

\bibitem{KLS3}
C.~E. Kenig, F.~Lin, and Z.~Shen.
\newblock Periodic homogenization of {G}reen and {N}eumann functions.
\newblock {\em Comm. Pure Appl. Math.}, 67(8):1219--1262, 2014.

\bibitem{MV}
S.~Moskow and M.~Vogelius.
\newblock First-order corrections to the homogenised eigenvalues of a periodic
  composite medium. {A} convergence proof.
\newblock {\em Proc. Roy. Soc. Edinburgh Sect. A}, 127(6):1263--1299, 1997.

\bibitem{BLtail}
C.~{Prange}.
\newblock Asymptotic analysis of boundary layer correctors in periodic
  homogenization.
\newblock {\em SIAM J. Math. Anal.}, 45(1):345--387, 2013.

\bibitem{ShenZhuge_arXiv16}
Z.~{Shen} and J.~{Zhuge}.
\newblock {Boundary Layers in Periodic Homogenization of Neumann Problems}.
\newblock {\em ArXiv e-prints}, Oct. 2016.

\end{thebibliography}

\end{document}